\documentclass[11pt, a4paper, reqno, english]{amsart}
\usepackage[margin=3cm]{geometry}

\author{Malte Borken}

\title[Quantitative rigidity of submanifolds]{A characterization of submanifolds of $\RR^{m \times n}$ satisfying optimal rigidity estimates}

\usepackage[english]{babel}

\usepackage{amsmath}
\usepackage{amssymb}
\usepackage{amsthm}
\usepackage{esint} 
\usepackage{mathtools}
\mathtoolsset{showonlyrefs}
\DeclarePairedDelimiter\abs{\lvert}{\rvert}
\DeclarePairedDelimiterXPP\Ln[1]{\mathcal{L}^n}(){}{#1}
\DeclarePairedDelimiter\norm{\lVert}{\rVert}

\DeclarePairedDelimiterX{\dualprod}[2]{\langle}{\rangle}{\,#1\, ,\,#2\,}
\DeclarePairedDelimiterX\set[1]\{\}{#1}%

\newcommand{\given}{:\,}

\usepackage[svgnames]{xcolor}
\usepackage[shortlabels]{enumitem}
\setlist[enumerate,1]{label=\textit{(\roman*)}}

\usepackage[T1]{fontenc}
\usepackage[utf8]{inputenc}
\usepackage{csquotes}
\usepackage[pdfpagemode=UseOutlines,colorlinks]{hyperref}

\usepackage{graphicx}

\usepackage[facing=yes,capbesidewidth=sidefil,capposition=beside,capbesideposition={top,outside},capbesidesep=qquad,floatrowsep=qquad]{floatrow}

\usepackage[font=small,labelfont=bf,labelsep=period]{caption}

\usepackage{aliascnt} 
\theoremstyle{plain}
\newtheorem{theorem}{Theorem}
\numberwithin{equation}{section}
\numberwithin{theorem}{section}

\newtheorem*{thmnn}{Theorem}

\newaliascnt{conjecture}{theorem}

\aliascntresetthe{conjecture}

\newaliascnt{proposition}{theorem}
\newtheorem{proposition}[proposition]{Proposition}
\aliascntresetthe{proposition}

\newaliascnt{lemma}{theorem}
\newtheorem{lemma}[lemma]{Lemma}
\aliascntresetthe{lemma}

\newaliascnt{corollary}{theorem}
\newtheorem{corollary}[corollary]{Corollary}
\aliascntresetthe{corollary}

\theoremstyle{remark}

\theoremstyle{definition}

\newaliascnt{definition}{theorem}
\newtheorem{definition}[definition]{Definition}
\aliascntresetthe{definition}

\newaliascnt{remark}{theorem}
\newtheorem{remark}[remark]{Remark}
\aliascntresetthe{remark}

\newaliascnt{example}{theorem}
\newtheorem{example}[example]{Example}
\aliascntresetthe{example}

\renewcommand{\AA}{\mathbb{A}}

\newcommand{\CC}{\mathbb{C}}

\newcommand{\NN}{\mathbb{N}}

\newcommand{\RR}{\mathbb{R}}

\newcommand{\ZZ}{\mathbb{Z}}

\newcommand{\cA}{\mathcal{A}}

\newcommand{\cE}{\mathcal{E}}
\newcommand{\cF}{\mathcal{F}}

\newcommand{\cL}{\mathcal{L}}
\newcommand{\cM}{\mathcal{M}}

\newcommand{\cP}{\mathcal{P}}
\newcommand{\cQ}{\mathcal{Q}}

\usepackage{mathrsfs}

\newcommand{\sP}{\mathscr{P}}

\newcommand{\wsc}{\overset{*}{\rightharpoonup}}

\newcommand{\mres}{\mathbin{\vrule height 1.6ex depth 0pt width 0.13ex\vrule height 0.13ex depth 0pt width 1.3ex}} 
\def\Xint#1{\mathchoice
{\XXint\displaystyle\textstyle{#1}}%
{\XXint\textstyle\scriptstyle{#1}}%
{\XXint\scriptstyle\scriptscriptstyle{#1}}%
{\XXint\scriptscriptstyle\scriptscriptstyle{#1}}%
\!\int}
\def\XXint#1#2#3{{\setbox0=\hbox{$#1{#2#3}{\int}$}
\vcenter{\hbox{$#2#3$}}\kern-.5\wd0}}

\def\dashint{\Xint-}

\newcommand{\del}{\partial}
\newcommand{\eps}{\varepsilon}

\newcommand{\ddt}{\frac{\mathrm{d}}{\mathrm{dt}}}

\newcommand{\Step}[1]{\medskip  \noindent \emph{Step #1.}}

\renewcommand{\skew}{\mathrm{skew}}
\newcommand{\sym}{\mathrm{sym}}

\newcommand{\son}{{\mathrm{SO}(n)}}
\DeclareMathOperator{\GL}{GL}
\newcommand{\conf}{{\mathrm{CO}_+(n)}}
\newcommand{\RRn}{{\mathbb{R}^n}}
\newcommand{\RRN}{{\mathbb{R}^N}}
\newcommand{\RRm}{{\mathbb{R}^m}}
\newcommand{\RRM}{{\mathbb{R}^M}}

\newcommand{\RRnn}{{\mathbb{R}^{n \times n}}}
\newcommand{\RRmn}{{\mathbb{R}^{m \times n}}}
\newcommand{\RRnnsk}{{\mathbb{R}^{n \times n}_\skew}}
\newcommand{\RRnnsym}{{\mathbb{R}^{n \times n}_\sym}}

\DeclareMathOperator{\supp}{supp}
\DeclareMathOperator{\Lip}{Lip}

\DeclareMathOperator{\dist}{dist}
\DeclareMathOperator{\diam}{diam}

\DeclareMathOperator{\Tr}{trace}

\DeclareMathOperator{\rank}{rank}

\DeclareMathOperator{\Lin}{L}
\DeclareMathOperator{\Gr}{Gr}

\DeclareMathOperator{\diag}{diag}
\DeclareMathOperator*{\osc}{osc}

\newcommand{\dd}{\mathrm{d}}
\newcommand{\dx}{\,\mathrm{d}x}

\newcommand{\dz}{\,\mathrm{d}z}
\newcommand{\dt}{\,\mathrm{d}t}

\newcommand{\loc}{\mathrm{loc}}
\newcommand{\op}{\mathrm{op}}
\newcommand{\Hom}{\mathrm{hom}}

\newcommand{\QR}{{QR}}

\newcommand{\id}{\mathrm{id}}
\newcommand{\idn}{\mathrm{Id}}

\newcommand{\ie}{i.e.\@ }

\begin{document}
\renewcommand{\sectionautorefname}{Section} 
\renewcommand{\subsectionautorefname}{Subsection}

\begin{abstract}
    Let $K \subset \RRmn$ be a compact $C^1$-submanifold with boundary, $p \in (1,\infty)$ and $Q\coloneq (0,1)^n$. We prove that $K$ satisfies an optimal quantitative rigidity estimate of the form
    \begin{equation}
        \norm{Du -  (Du)_{Q}}_{L^p} \leq C \norm{\dist_K(Du)}_{L^p} \qquad \forall u \in W^{1,p}(Q,\RRm),
    \end{equation}
    \emph{if and only if} $K$ satisfies sequential rigidity and for each $A \in K$,
    the  tangent space to $K$ at $A$ satisfies exact rigidity.
    We further prove that this rigidity estimate is stable under small graphical perturbations of $K$.
    The key technical ingredient is proving that outside some small \enquote{bad} set where the maximal function of $\dist_K^p(Du)$ is large,
    the size of the superlevel sets of $\abs{Du - (Du)_Q}$ decays exponentially.
    This is achieved by an adaptation of the John-Nirenberg inequality for $BMO$.
\end{abstract}
\maketitle

\hypersetup{allcolors=MediumBlue}

\section{Introduction}
A classical result known as Liouville's theorem states that the set $\son$ satisfies the following \emph{rigidity} property:
If $\Omega \subset \RRn$ is open, bounded and connected, and  $u:\Omega \to \RRn$ is a smooth vector field satisfying the \emph{differential inclusion} $Du \in \son$, that is,
\begin{equation}
    Du(x) \in \son \qquad \text{for a.e. }x \in \Omega,
\end{equation}
then $u$ must be affine.
This was later improved by \textsc{Reshetnyak} \cite{MR218544}, who proved as a byproduct of his work on quasiconformal mappings that
$\son$ satisfies what is nowadays called \emph{sequential rigidity}, that is, for every sequence of vector fields $(u_j)_{j \in \NN} \subset W^{1,n}(\Omega;\RRn)$ satisfying $\dist_{\son}(Du_j) \to 0$ in $L^n$, there exists a matrix $A \in \son$ and a subsequence
$(u_{k_j})_{j\in \NN}$, such that $Du_{k_j} \to A$ in $L^n$.
In applications, it is often necessary to be able to quantify this convergence.
Substantial progress in this direction was made by \textsc{John} \cite{John61, John72} and later \textsc{Kohn} \cite{Kohn82}, before the problem was finally resolved by \textsc{Friesecke, James and Müller} \cite{FJM} who proved the following
\emph{optimal rigidity estimate}, also known as geometric rigidity:
\begin{thmnn}[{Geometric rigidity, \cite[Thm.~3.1]{FJM}}]
    Let $\Omega\subset \RRn$ be a bounded Lipschitz domain.
    Then there exists a constant $C=C(\Omega)>0$ such that
    \begin{equation} \label{eq:fjm}
        \inf_{R \in \son} \int_\Omega \abs{ D  u - R}^2 \dx \leq C \int_{\Omega} \dist^2_{\son}( D  u) \dx \qquad \forall u \in W^{1,2}(\Omega;\RRn).
    \end{equation}
\end{thmnn}
The significance of their result from the viewpoint of applications is due to the fact that in continuum mechanics,
one usually assumes that the constitutive laws corresponding to a material satisfy frame indifference,
which roughly expresses that the laws of physics should not depend on the position and orientation of the observer.
This results in a certain $\son$-invariance, for example of the elastic energy density corresponding to an elastic material.
The estimate \eqref{eq:fjm} is then often needed to deal with this invariance.

For instance, \eqref{eq:fjm} has been used to rigorously derive by means of $\Gamma$-convergence models for lower dimensional elastic objects such as plates, shells and rods \cite{FJM, FJM06, FJMM03, MR2018669}, 
see also the surveys \cite{DimRedMüller, Lewicka23}; and further to derive linear elasticity from nonlinear elasticity \cite{DalMasoNegriPercivale}.
There have been numerous generalizations, including mixed growth conditions \cite{CDM13}; inequalities for incompatible fields (i.e.\ vector fields which are not gradient fields) \cite{MR3283554, MR4237631};
and multiwells (i.e.\ sets of the form $\bigcup_{i=1}^N \son A_i$), both in the compatible \cite{MR2217607, MR2684308, MR3089738} and incompatible case \cite{MR2039459, MR2253059}.
Recently, \textsc{Conti, Dolzmann and Müller} \cite{ContiDolzmannMüller24} also obtained a version of \eqref{eq:fjm} for endomorphisms of a compact Riemannian manifold.

While \eqref{eq:fjm} and its generalizations continue to be highly successful in applications, they do not give much insight into exactly which structural properties of the set $\son$ enable it to satisfy the estimate \eqref{eq:fjm}.
This is addressed in the present work.
More precisely, we consider the following notions of rigidity:
\begin{definition} \label{def:rigidity}
    Let $H \subset \RRmn$ be closed and $p\in(1,\infty)$. Denote $Q := (0,1)^n$.
    \begin{enumerate}
        \item We say that $H$ satisfies \emph{exact rigidity} if every map $u \in W^{1,\infty}(Q;\RRm)$ satisfying the differential inclusion $Du \in H$ is affine.
        \item We say that $H$ satisfies \emph{sequential rigidity} if for every weak-* convergent sequence $(u_j)_{j \in \NN} \subset W^{1,\infty}(Q;\RRm)$ satisfying
            \begin{equation}
                \dist_H(Du_j) \to 0 \quad \text{in measure},
            \end{equation}
            there exists $A \in H$ such that
            \begin{equation}
                Du_j \to A \quad \text{in measure}.
            \end{equation}
        \item We say that $H$ satisfies \emph{$L^p$ quantitative rigidity} if there exists $C>0$ such that
            \begin{equation} \label{eq:intro rig est}
                \norm{Du - (Du)_Q}_{L^p} \leq C \norm{\dist_H(Du)}_{L^p} \qquad \forall u \in W^{1,p}(Q;\RRm),
            \end{equation}
            where $(Du)_Q$ denotes the average of $Du$ on $Q$.
            Furthermore, we denote by $C_{QR}$ the \emph{optimal constant} in \eqref{eq:intro rig est}, that is
            \begin{equation} \label{eq:defCQR}
                C_{QR} := C_\QR(H,p) := \sup_{\substack{u\in W^{1,p}(Q;\RRm)\\ \norm{\dist_H(Du)}_{L^p}>0}} \frac{\norm{Du - (Du)_Q}_{L^p}}{\norm{\dist_H(Du)}_{L^p}} \in [1,\infty].
            \end{equation}
    \end{enumerate}
\end{definition}
Note the following two consequences of this definition.
First, for any $p\in(1,\infty)$ and closed set $H \subset \RRmn$ the following implications hold:
\begin{equation}
    L^p \text{ quantitative rigidity} \; \implies \; \text{sequential rigidity} \; \implies \; \text{exact rigidity}.
\end{equation}
Second, by the triangle inequality, it follows that replacing the left-hand side of \eqref{eq:intro rig est} by either $\,\inf_{A \in \RRmn} \norm{Du - A}_{L^p}\,$ or $\,\inf_{A \in H} \norm{Du - A}_{L^p}\,$ results in an equivalent definition of $L^p$ quantitative rigidity.
In particular, \eqref{eq:fjm} states that $\son$ satisfies $L^2$ quantitative rigidity.

The main question we are interested in is the following:
\emph{What are necessary and sufficient conditions for a compact submanifold $K \subset \RRmn$ to satisfy $L^p$ quantitative rigidity?}
Up to now, apart from special cases such as (linear) subspaces, finite sets of matrices, and multiwells of the form $\bigcup_{i=1}^N \son A_i$, there seem to be few results in the direction of answering this question.
We particularly point out the recent work of \textsc{Lamy, Lorent and Peng} \cite{MR4735627}, who showed that a smooth, compact and connected one-dimensional submanifold without boundary
$K \subset \RR^{2 \times 2}$ satisfies $L^2$ quantitative rigidity if and only if there exists a constant $C_*>0$ such that
\begin{equation} \label{eq:LLPcondition}
    \abs{A-B}^2 \leq C_* \abs{\det(A -B)} \qquad \forall A,B \in K.
\end{equation}
In order to understand \eqref{eq:LLPcondition}, it is helpful to first consider conditions that are \emph{necessary} for $L^p$ quantitative rigidity:
If $K \subset \RRmn$ is a compact $C^1$-\emph{submanifold} (possibly with boundary) satisfying $L^p$ quantitative rigidity, then it holds:
\begin{enumerate}[\textit{(\alph*)}]
    \item \label{intro(a)} $K$ satisfies sequential rigidity;
    \item \label{intro(b)} For each $A \in K$, the tangent space $T_AK$ satisfies exact rigidity.
\end{enumerate}
Condition \ref{intro(a)} follows directly from the definition, while \ref{intro(b)} follows by a blow-up type argument, see \autoref{lem:QuantRigImpliesKorn}.
It is well known that for a closed set $H \subset \RRmn$ to satisfy exact rigidity, it is necessary that $H$ does not
contain any \emph{rank-one connections}, that is, there exist no matrices $A,B \in H$ such that $\rank(A-B) = 1$ (see, e.g., \cite{MR1731640}).
Thus, \ref{intro(a)} and \ref{intro(b)} further imply:
\begin{enumerate}[\textit{(\alph*)}$'$]
    \item \label{intro(a)'} $K$ contains no rank-one connections;
    \item \label{intro(b)'} For every $A \in K$, the tangent space $T_AK$ contains no rank-one matrices.
\end{enumerate}
The inequality \eqref{eq:LLPcondition} can be understood as a quantitative way of capturing these weakened conditions \ref{intro(a)'} and \ref{intro(b)'}, and is in fact equivalent to them in the setting of compact one-dimensional submanifolds without boundary of $\RR^{2 \times 2}$.
However, as further shown by \textsc{Lamy, Lorent and Peng} \cite{MR4735627}, the phenomenon that \ref{intro(a)'} and \ref{intro(b)'} even imply the necessary conditions \ref{intro(a)} and \ref{intro(b)} is exclusive to low dimensions, and already fails in $\RR^{3 \times 3}$.

\subsection{Main results}
The first main result states that in general, \ref{intro(a)} and \ref{intro(b)} are still sufficient, thus providing \emph{a full characterization of compact submanifolds of $\RRmn$ satisfying $L^p$ quantitative rigidity}:
\begin{theorem} \label{thm:GeneralRigidity}
    Let $K \subset \RRmn$ be a compact $C^1$-submanifold with boundary
    and $p \in (1,\infty)$.
    Then $K$ satisfies $L^p$ quantitative rigidity if and only if \ref{intro(a)} and \ref{intro(b)} hold.
\end{theorem}
The second main result states that \emph{$L^p$ quantitative rigidity
is stable under small graphical perturbations}:
\begin{theorem} \label{thm:Stability}
    Let $K \subset \RRmn$ be a compact $C^1$-submanifold with boundary satisfying $L^p$ quantitative rigidity,
    where $p \in (1,\infty)$.
    Then there exist constants $\rho(K,p) > 0$ and $C(K,p) > 0$ such that the following holds.
    If $\Phi:\RRmn \to \RRmn$ is bi-Lipschitz with
    \begin{equation} \label{eq:LipCondonPhi}
        \Lip(\Phi - \id) \leq \rho,
    \end{equation}
    then $\tilde K\coloneq \Phi(K)$ satisfies $L^p$ quantitative rigidity with
    \begin{equation}
        C_{QR}(\tilde K,p) \leq C.
    \end{equation}
    In particular, $\tilde K$ satisfies $L^p$ quantitative rigidity.
\end{theorem}
Since \ref{intro(a)} and \ref{intro(b)} do not depend on $p$, we obtain the following immediate consequence of \autoref{thm:GeneralRigidity}.
\begin{corollary} \label{cor:Invariance-p-domain}
    Let $K\subset \RRmn$ be a compact $C^1$-submanifold with boundary and $p,q \in (1,\infty)$.
    Then $K$ satisfies $L^p$ quantitative rigidity if and only if it satisfies $L^q$ quantitative rigidity.
\end{corollary}
In the known cases of sets that satisfy $L^p$ quantitative rigidity \cite{FJM, MR2039459, MR2253059, MR4735627}, 
the validity of \ref{intro(a)} and \ref{intro(b)} is readily verified, or at least has been known much longer than the corresponding quantitative result.
Thus, \autoref{thm:GeneralRigidity} presents a way of proving these results with a single, unified approach.

It also enables finding possible new examples of quantitatively rigid sets.
Particularly, in \autoref{prop:null lagrangian ex} we extend the result of \textsc{Lamy, Lorent and Peng} \cite{MR4735627}: 
If a compact, $1$-dimensional $C^1$-submanifold with boundary $K \subset \RRmn$ satisfies an inequality similar to \eqref{eq:LLPcondition}, 
then it must satisfy \ref{intro(a)} and \ref{intro(b)}, and therefore $L^p$ quantitative rigidity.
From this perspective, \autoref{thm:Stability} presents another way of generating a large class of examples.

\subsection{Structure of the paper}
In \autoref{sec:prelim} we introduce some notation, recall the definition and basic properties of gradient Young measures and quasiconvexity, as well as their relation to sequential rigidity,
and discuss the connection between $\CC$-ellipticity and subspaces that satisfy exact rigidity.
Sections \ref{sec:GeneralRigidity} and \ref{sec:Stability} are devoted to the proofs of our main results, \autoref{thm:GeneralRigidity} and \autoref{thm:Stability}, respectively.
In \autoref{sec:Examples} we discuss examples and the relation to other rigidity results.
Finally, in \autoref{sec:generalDomains} we show how the rigidity estimate \eqref{eq:intro rig est} can be extended to  domains other than the unit cube.

\section{Preliminaries and Notation} \label{sec:prelim}
\subsection{General notation} \label{subsec:generalNotation}

Let $X$ be a normed space. For $x \in X$ and $r\geq 0$, we denote
by $B(x,r)$ the open ball of radius $r$ around $x$. If  $X=\RRn$, we set
$Q(x,r)\coloneq x+(-r,r)^n$ and $Q\coloneq (0,1)^n$.
Given two sets $H,G \subset X$, we write $\dist(H,G)\coloneq  \inf_{x\in H,y\in G}\abs{x-y}$, and $\dist_H(x)\coloneq \dist(\set{x},H)$.
The identity map on $X$ is denoted by $\id$.
The letter $C$ is reserved to denote an absolute constant, which is allowed to change from line to line, and depends only on the dimensions $n$ and $m$.
Dependence on any additional parameters such as $K$ and $p$ is expressed by writing $C(K,p)$ or $C=C(K,p)$.
If the need arises to refer to a specific constant, we may label it using subscripts.

Furthermore, for $k \leq N \in \NN$, we write $\Gr(k,\RRN)$ for the Grassmannian manifold of all $k$-dimensional subspaces of $\RRN$.
If $V \in \Gr(k,\RRN)$, we denote by $\pi_V:\RRN \to V$ the orthogonal projection onto $V$.
We endow $\Gr(k,\RRN)$ with the metric
\begin{equation} \label{eq:DefMetricGrassmannian}
    d(V,W) \coloneq  \abs{\pi_V -\pi_W}_\op = \abs{\pi_{V^\perp} -\pi_{W^\perp}}_\op, \qquad V,W \in \Gr(k,\RRN).
\end{equation}

\subsection{Gradient Young measures and quasiconvexity} \label{sec:GY}
Gradient Young measures and quasiconvexity are essential tools for studying the conditions \ref{intro(a)} and \ref{intro(b)} of \autoref{thm:GeneralRigidity}; let us briefly recall the definition and basic properties.

We restrict ourselves to the setting of Lipschitz functions; gradient Young measures as defined below are often also referred to as $L^\infty$-gradient Young measures, in contrast
to gradient Young measures generated by sequences $(Du_j)_{j \in \NN} \subset L^p(\Omega;\RRmn)$, where $p \in [1,\infty)$.

The basic result guaranteeing existence of gradient Young measures is the following, see for example \cite[Thm.~4.1]{Rindl}.
Let $(u_j)_{j \in \NN}\subset W^{1,\infty}(\Omega;\RRm)$ be a bounded sequence. Then there exist a (non relabeled) subsequence
and a family of probability measures
\begin{equation}
    \nu = (\nu_x)_{x\in\Omega} \subset \cP(\RRmn),
\end{equation}
depending measurably on $x$, such that
for all $f \in C(\RRmn)$ it holds
\begin{equation}
    f(Du_j) \wsc \dualprod{\nu_\cdot}{f} \quad \text{in } L^\infty.  \label{eq:DefGradYoungMeasure}
\end{equation}

We call $\nu$ the \emph{gradient Young measure generated by the (sub-)sequence} $(Du_j)_{j \in \NN}$ (or, with slight abuse of notation, \emph{by the sequence $(u_j)_{j \in \NN}$}),
and denote by $GY(\Omega;\RRmn)$ the set of all measures generated in this way.
Testing with $f = \id$ in \eqref{eq:DefGradYoungMeasure},
it follows that $u_j$ converges weakly* in $W^{1,\infty}(\Omega;\RRm)$ to a function $u$
whose gradient is almost everywhere given by the \emph{barycenter of $\nu_x$},
\begin{equation}
    Du(x) = [\nu_x] \coloneq  \int_\RRmn A \,\dd\nu_x(A) \qquad \text{at a.e. } x \in \Omega.
\end{equation}
We call $u$ the \emph{underlying deformation} of $\nu$.

If a gradient Young measure $(\nu_x)_{x \in \Omega} \in GY(\Omega;\RRmn)$ is constant a.e.\ in $x$ (i.e.\ $\nu_x = \nu$ for some $\nu \in \cP(\RRmn)$ and a.e.\ $x \in \Omega$),
we call it a \emph{homogeneous gradient Young measure}, and write $\nu \in GY_\Hom(\RRmn)$.

A notion that is closely related to gradient Young measures is that of quasiconvexity:
We call a locally bounded and measurable function $f:\RRmn \to \RR$ \emph{quasiconvex} if for every $\eta \in W^{1,\infty}_0(Q;\RRm)$ and every $A \in \RRmn$ it holds
\begin{equation} \label{eq:DefQC}
    f(A) \leq \dashint_Q f(A + D\eta) \dx.
\end{equation}
Recall that any quasiconvex function is rank-one convex and therefore locally Lipschitz continuous, see \cite[Lemmas~5.3 \& 5.6]{Rindl}.
Furthermore, every homogeneous gradient Young measure admits a generating sequence with linear boundary values (see, e.g., \cite[Lem.~4.13]{Rindl}), such that \eqref{eq:DefQC} immediately extends to $GY_\Hom(\RRmn)$:
If $f:\RRmn \to \RR$ is quasiconvex and $\nu \in GY_\Hom(\RRmn)$, then it holds
\begin{equation} \label{eq:GYhomQC}
    f([\nu]) \leq \int_\RRmn f \,\dd \nu.
\end{equation}
Conversely, if $f:\RRmn \to \RR$ is a locally bounded and measurable function satisfying \eqref{eq:GYhomQC} for all $\nu \in GY_\Hom(\RRmn)$,
it follows that $f$ is quasiconvex (\cite[Lem.~4.14]{Rindl}).

We call $f:\RRmn \to \RR$ \emph{quasiaffine} if both $f$ and $-f$ are quasiconvex.
It is well known that the determinant, $\det:\RRnn \to \RR$, is quasiaffine.
Furthermore, we say that $f$ \emph{strictly quasiconvex} if there exists a strictly convex
$g:\RRmn \to [0,\infty)$ such that $f-g$ is still quasiconvex.
If the function $g$ is of the form $\lambda \abs{\cdot}^2$, we say that $f$ is \emph{{$\lambda$-uniformly} quasiconvex}.

\subsection{A characterization of sequential rigidity} \label{subsec:charseq}
It is well-known that sequential rigidity can be characterized using gradient Young measures.
To this end, let us introduce yet another notion of rigidity.
\begin{definition}
    Let $H \subset \RRmn$ be closed. We say that $H$ satisfies \emph{homogeneous sequential rigidity} if every homogeneous gradient Young measure supported in $H$ is trivial, that is, if $\nu \in GY_\Hom(\RRmn)$  satisfies
    \begin{equation}
        \supp\nu \subset H,
    \end{equation}
    there exists some $A \in H$ such that
    \begin{equation}
        \nu = \delta_{A}.
    \end{equation}
\end{definition}
Using this, sequential rigidity can now be characterized as follows.
\begin{proposition} \label{prop:seqRig}
    Let $K \subset \RRmn$ be compact and $p\in (1,\infty)$. 
    Then the following are equivalent:
    \begin{enumerate}
        \item \label{prop:seqRig(i)} $K$ satisfies sequential rigidity.
        \item \label{prop:seqRig(ii)} For every gradient Young measure $\nu \in GY(Q;\RRmn)$ satisfying
            \begin{equation}
                \supp\nu_x \subset K \qquad \text{for a.e. } x \in Q,
            \end{equation}
            there exists $A \in K$ such that
            \begin{equation}
                \nu_x = \delta_{A} \qquad \text{for a.e. } x \in Q.
            \end{equation}
        \item \label{prop:seqRig(iii)} $K$ satisfies exact rigidity and homogeneous sequential rigidity.
        \item \label{prop:seqRig(iv)} For every $\eps > 0$,
            there exists $\delta=\delta(\eps,p)>0$ such that for all $u \in W^{1,p}(Q;\RRm)$, it holds
            \begin{equation}
                \norm{\dist_K(Du)}_{L^p} \leq \delta \qquad \implies \qquad \norm{Du - (Du)_Q}_{L^p} \leq \eps.
            \end{equation}
    \end{enumerate}
\end{proposition}
\begin{proof}
A proof of the equivalence of \ref{prop:seqRig(i)}-\ref{prop:seqRig(iii)} can be found in \cite[Lem.~8.8~(ii)]{Rindl}.
Furthermore, it is clear that \ref{prop:seqRig(iv)} implies \ref{prop:seqRig(i)}.
The converse direction follows by contradiction and truncation \cite[Prop.~A.1]{FJM}.
\end{proof}
In practice, it is often easiest to prove \ref{prop:seqRig(iii)}.
A sufficient condition for homogeneous sequential rigidity is the following.

\begin{remark} \label{rem:UnifQCsets}
    Assume that $K \subset f^{-1}(0)$ for some nonnegative strictly quasiconvex function $f:\RRmn \to [0,\infty)$.
    Then $K$ satisfies homogeneous sequential rigidity:
    Indeed, let $g:\RRmn \to \RR$ be strictly convex such that $f-g$ is still quasiconvex.
    Now, if $\nu \in GY_\Hom(\RRmn)$ satisfies
    $\supp \nu \subset K$, it follows
    \begin{equation}
        -g([\nu]) \leq f([\nu]) - g([\nu]) \leq \int_\RRmn f -g \,\dd\nu = \int_\RRmn -g \,\dd\nu.
    \end{equation}
    Since $g$ is strictly convex, this can only happen if $\nu = \delta_A$ for some $A \in \RRmn$.
\end{remark}

\subsection{Rigidity of subspaces and \texorpdfstring{$\CC$-}{}ellipticity} \label{sec:subspaces}

In order to prove one of the directions in \autoref{thm:GeneralRigidity}, we need to pass from the two qualitative conditions \ref{intro(a)} and \ref{intro(b)} to a quantitative statement, $L^p$ quantitative rigidity.
This happens first on the level of subspaces, where exact rigidity implies that a certain linear differential operator satisfies a very strong type of ellipticity, $CC$-ellipticity, which we define below.  
The corresponding $CC$-ellitic estimate is equivalent to the subspace satisfying $L^p$ quantitative rigidity.
This subsection is dedicated to this preliminary result.

Let $\cA$ be an \emph{$l$-homogeneous, linear, constant-coefficient differential operator on $\RRn$ between $\RRm$ and $\RRM$}, that is
\begin{equation} \label{eq:formofA}
    \cA u \coloneq   \sum_{\substack{\alpha \in \NN^n\\ \abs{\alpha }= l} } \cA_\alpha \del^\alpha u, \qquad u:\Omega \subset \RRn \to \RRm,
\end{equation}
where each $\cA_\alpha \in \Lin(\RRm;\RRM)$, and we use the notation
\begin{equation}
    \abs{\alpha}\coloneq \sum_{j \leq n} \alpha_j, \qquad \del^\alpha \coloneq  \prod_{j \leq n}\del_j^{\alpha_j}, \qquad \xi^\alpha \coloneq  \prod_{j \leq n}\xi_j^{\alpha_j},  \qquad \alpha \in \NN^n, \; \xi \in \CC^n.
\end{equation}
For $\xi \in \CC^n$, we denote by
\begin{equation}
    \AA(\xi) \coloneq \sum_{\abs{\alpha}=l} \xi^\alpha\cA_\alpha \in \Lin(\CC^m;\CC^M)
\end{equation}
the \emph{symbol} of $\cA$.
We call $\cA$ \emph{elliptic}, or \emph{$\RR$-elliptic}, if
\begin{equation}
    \AA(\xi)b \neq 0 \qquad \forall \xi \in \RR^n \setminus \set{0},\; b \in  \RR^m \setminus \set{0};
\end{equation}
and \emph{$\CC$-elliptic} if
\begin{equation}
    \AA(\xi)b \neq 0 \qquad \forall \xi \in \CC^n \setminus \set{0},\; b \in  \CC^m \setminus \set{0}.
\end{equation}
Note that $\CC$-ellipticity is a much stronger property than $\RR$-ellipticity and in many cases too restrictive, see  \autoref{ex:HarmonicGradients} below.
For us, the strength of $\CC$-ellipticity lies in the following results due to \textsc{Smith} \cite{MR282046}, see also \cite[Lemma 3.1, Theorem B]{MR4711552}.
\begin{proposition} \label{prop:CC-ellipticityIsNice}
    Let $\cA$ be of the form \eqref{eq:formofA}.
    The following are equivalent:
    \begin{enumerate}
        \item \label{prop:CCellNice(i)} $\cA$ is $\CC$-elliptic.
        \item \label{prop:CCellNice(ii)} There exists a number $L \in \NN$ and an $(L-l)$-homogeneous, linear, constant-coefficient differential operator $\cE$ such that
            \begin{equation}
                \cE \circ \cA = D^L.
            \end{equation}
        \item \label{prop:CCellNice(iii)} For any open and connected set $\Omega \subset \RRn$, the \emph{kernel of $\cA$},
            \begin{equation}
                \ker \cA := \set{v \in W^{l,1}(\Omega;\RRm) \given \cA v = 0},
            \end{equation}
            is finite dimensional.
    \end{enumerate}
    Moreover, if the above are true, then $\ker \cA$ consists of $\RRm$-valued polynomials and does not depend on $\Omega$; furthermore for $p \in (1,\infty)$ it holds
    \begin{equation} \label{eq:CC-Ell-Korn}
        \min_{\eta \in \ker \cA} \norm{u - \eta}_{W^{l,p}} \leq C(\cA,p,\Omega) \norm{\cA u}_{L^p} \qquad \forall u \in W^{l,p}(\Omega;\RRm).
    \end{equation}
\end{proposition}

\begin{definition}
    Let $V \in \Gr(k,\RRmn)$, and recall that we denote by $\pi_{V^\perp}:\RRmn \to V^\perp$ the orthogonal projection onto $V^\perp$.
    We associate with $V$ the first order, homogeneous, linear, constant-coefficient differential operator
    \begin{align}
        \cA^V u \coloneq  \pi_{V^\perp}Du,
    \end{align}
    whose symbol is given by
    \begin{equation}
        \AA^V(\xi)b=  \pi_{V^\perp}(b \otimes \xi).
    \end{equation}
    We call $V \in \Gr(k,\RRmn)$ \emph{$\RR$-elliptic} (resp.\@ \emph{$\CC$-elliptic}) if $\cA^V$ is $\RR$-elliptic (resp.\@ $\CC$-elliptic).
\end{definition}
Note that $u \in W^{1,1}_\loc(\Omega;\RRm)$ satisfies the differential inclusion $Du \in V$ if and only if it solves $\cA^V u = 0$.
Assume now that $V \in \Gr(k,\RRmn)$ satisfies exact rigidity. 
Then
\begin{equation}
    \ker \cA^V = \set{(x \mapsto Ax + b )\given A \in V, \; b \in \RRm},
\end{equation}
and in particular $\ker \cA^V$ is finite dimensional.
By \autoref{prop:CC-ellipticityIsNice}, it follows that $\cA^V$ is $\CC$-elliptic and satisfies \eqref{eq:CC-Ell-Korn}.
Since $\abs{\cA^V u} = \dist_V(Du)$, this is equivalent to $V$ satisfying $L^p$ quantitative rigidity for any $p \in (1,\infty)$.
Let us summarize this for later reference.
\begin{corollary} \label{cor:Subspace-QuantIffExact}
    Let $V \in \Gr(k,\RRmn)$ and $p \in (1,\infty)$. Then $V$ satisfies $L^p$ quantitative rigidity if and only if it satisfies exact rigidity.
\end{corollary}
\begin{remark} \label{rem:RR-Ellipticity}
    In comparison, $\RR$-elliptic subspaces can be characterized as follows, see \cite[Lem.~2.7]{MR1731640}.
    For $V \in \Gr(k,\RRmn)$ the following are equivalent:
    \begin{enumerate}
        \item $V$ contains no rank-one matrices.
        \item $V$ is $\RR$-elliptic.
        \item For every $p \in (1,\infty)$ there exists a constant $C(V,p)>0$ such that
            \begin{equation} \label{eq:R-ellipticEst}
                \norm{Du}_{L^p} \leq C \norm{\dist_V(Du)}_{L^p} \qquad \forall u \in W^{1,p}(\RRn;\RRm).
            \end{equation}
        \item Every homogeneous gradient Young measure supported in $V$ is a Dirac measure.
        \item Every solution to the differential inclusion $Du \in V$ is smooth (as a matter of fact, even analytic).
    \end{enumerate}
    For our purposes, the estimate \eqref{eq:R-ellipticEst} would not be strong enough, as we cannot afford to make restrictions on the boundary data of the functions considered.
\end{remark}

We end the subsection with a few examples and remarks.
It is well-known that the \emph{skew-symmetric matrices},
\begin{equation}
    \RRnnsk \coloneq \set{A \in \RRnn \given A = -A^T} \in \Gr(n(n-1)/2,\RRnn),
\end{equation}
satisfy exact rigidity.
The corresponding quantitative statement given by \autoref{cor:Subspace-QuantIffExact} is known as \emph{Korn's inequality}.

The following two examples show that $\RR$-ellipticity, $\CC$-ellipticity and exact rigidity are generally not equivalent for subspaces.
\begin{example} \label{ex:HarmonicGradients}
    Let $n\geq 2$ and define
    \begin{equation}
        V_1 \coloneq  \set{A \in \RRnnsym \given \Tr A = 0} \in \Gr\bigg(\frac{n(n+1)}{2}-1,\RRnn\bigg).
    \end{equation}
    It is easily seen that $Du \in V_1$ is equivalent to $u$ locally being the gradient of a harmonic function (and in particular smooth).
    Thus, by \autoref{rem:RR-Ellipticity}, it follows that $V_1$ is $\RR$-elliptic.
    However, \autoref{prop:CC-ellipticityIsNice} shows that $V_1$ is not $\CC$-elliptic, since the space of gradients of harmonic functions defined on $\RRn$ is infinite dimensional.
    Alternatively, this can also be checked on the level of the symbol.
\end{example}
\begin{example} \label{ex:TangentspaceMoebius}
    Let $n \geq 3$ and set
    \begin{equation}
        V_2 \coloneq  V_1^\perp = \RRnnsk + \RR \idn \in \Gr\bigg(\frac{n(n-1)}{2}+1,\RRnn\bigg).
    \end{equation}
    Then $V_2$ is $\CC$-elliptic, but does not satisfy exact rigidity.
    Indeed, it can be checked (see, e.g., \cite[Ch.~3, §2.2]{MR1326375}) that $\ker \cA^{V_2}$ is given by all polynomials $\eta$ of the form
    \begin{equation}
        \eta(x) = Bx + \lambda x + a + b\abs{x}^2 -2 (x \cdot b)x, \qquad B \in \RRnnsk,\, \lambda \in \RR,\, a,b \in \RRn.
    \end{equation}

\end{example}

Nevertheless, there is one case where $\RR$-ellipticity, $\CC$-ellipticity and exact rigidity are in fact equivalent.
\begin{lemma} \label{lem:1d-Subspace}
    Let $V = \RR A \subset \RRmn$ be a $1$-dimensional subspace. Then $V$ satisfies exact rigidity if and only if $\rank A > 1$.
\end{lemma}
\begin{proof}
Assume that $\rank A > 1$. A suitable transformation and the fact that exact rigidity does not depend on the domain allow us to reduce to the case $m=n=2$ and $A= \idn$.
Now assume $u \in W^{1,\infty}(Q;\RR^2)$ satisfies $Du \in V$. Then, since $Du$ is diagonal, it follows that $u_j = u_j(x_j)$, $j = 1,2$. Hence,
\begin{equation}
    Du(x) =
    \begin{pmatrix}
        \del_1u_1(x_1) & 0\\
        0 & \del_2 u_2(x_2)
    \end{pmatrix} = \lambda(x)\idn,
\end{equation}
which implies that $Du$ must be constant.
\end{proof}

\section{Proof of {\autoref*{thm:GeneralRigidity}}} \label{sec:GeneralRigidity}
The aim of this section is to prove \autoref{thm:GeneralRigidity}.
To this end, we start by introducing some notation and collecting a few preliminary lemmas.
\begin{definition}[$\overline \delta_K$ and $\underline \delta_K$] \label{def:delta_K}
    Let $N,k \in \NN$ and let $K \subset \RRN$ be a closed set.
    Then for $a \in K$, $V \in \Gr(k,\RRN)$ and $\eps > 0$, we set $\overline \delta_K(a,V,\eps)$ to be
    the largest $\overline \delta \geq 0$ such that
    \begin{equation} \label{eq:UnifTangEst}
        \begin{gathered}
            \dist_V(x-a) \leq \dist_K(x) + \eps\abs{x-a}\\
        \end{gathered} \qquad \forall x \in B(a,\overline \delta).
    \end{equation}
    We also set $\underline \delta_K(a,V,\eps)$ to be the largest $\underline \delta \geq 0$ such that
    \begin{equation} \label{eq:ReverseTangEst}
        \dist_K(x) \leq \dist_V(x-a) + \eps\abs{x-a} \qquad \forall x \in B(a,\underline \delta).
    \end{equation}
\end{definition}
\begin{remark} \label{rem:after bar delta}
    Alternatively, $\overline \delta_K(a,V,\eps)$ can be characterized as follows.
    For $x \in \RRN$, denote $r(x)\coloneq (\dist_V(x-a)-\eps\abs{x-a})_+$.
    Then for $\delta > 0$, the condition
    \begin{equation}
        \overline \delta_K(a,V,\eps) \geq \delta
    \end{equation}
    is equivalent to $K \cap U = \emptyset$, where
    \begin{equation}
        U \coloneq  \bigcup_{x \in B(a,\delta)}B(x,r(x)),
    \end{equation}
    with the convention $B(x,0)=\emptyset$.
    Consider the cone
    \begin{equation}
        I(V,\eps,a)\coloneq  r^{-1}(0) = \set{x \in \RRN \given \dist_V(x-a) \leq \eps\abs{x-a}}.
    \end{equation}
    A short computation shows that
    \begin{equation}
        B(a,\delta) \setminus I(V,\eps,a) \subset U \subset B(a,2\delta) \setminus I(V,\eps,a),
    \end{equation}
    and thus,
    \begin{equation}
        K \cap B(a,2\delta) \subset I(V,\eps,a) \quad \Rightarrow \quad \overline \delta_K(a,V,\eps) \geq \delta \quad \Rightarrow \quad K \cap B(a,\delta) \subset I(V,\eps,a).
    \end{equation}
    A similar characterization is holds for $\underline \delta_K(a,V,\eps)$.
\end{remark}
\begin{figure}[t]
    \ffigbox
    {\caption{A sketch of the sets described in \autoref{rem:after bar delta}.}}
    {\includegraphics[trim={5.3cm 8.6cm 10.7cm 3.5cm},clip,width=0.6\textwidth]{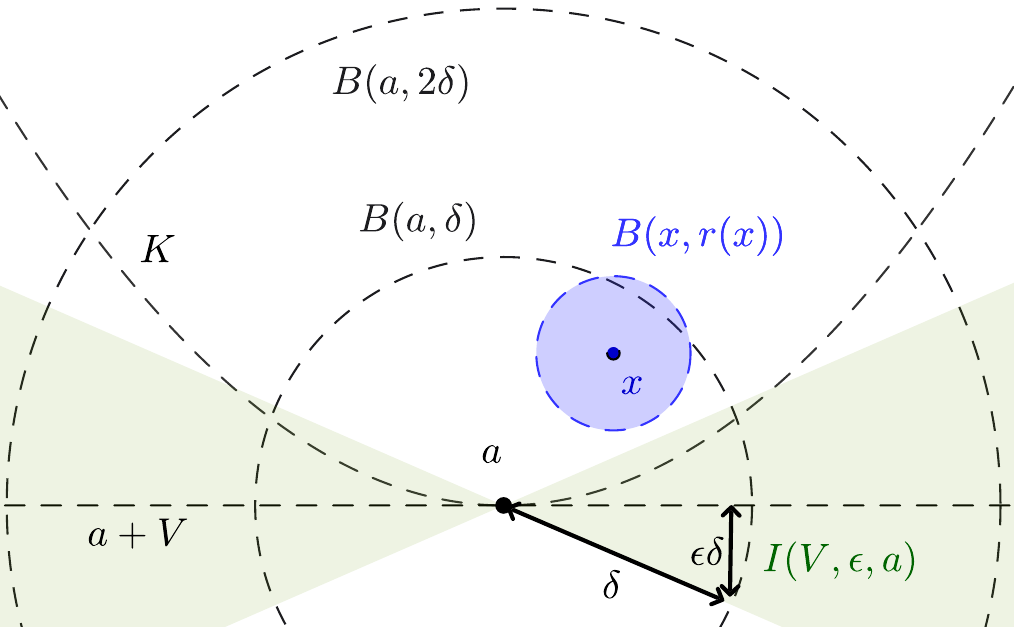}}
\end{figure}
From this point of view, the reason for working with compact $C^1$-submanifolds with boundary is that they behave well with respect to the quantities $\overline \delta_K$ and $\underline \delta_K$.
This is the content of the following elementary lemma, whose proof we omit.
\begin{lemma} \label{lem:PropertiesMFLD}
    Let $K \subset \RRN$ be a $k$-dimensional $C^1$-submanifold with boundary.
    Then the following hold.
    \begin{enumerate}
        \item The tangent space to $K$ at a point  $a \in K$ can be characterized as the unique $k$-dimensional subspace $V \in \Gr(k,\RRN)$ such that \label{lem:PropertiesMFLD(i)}
            \begin{equation} \label{eq:CharTanSpace}
                \overline \delta_K(a,V,\eps) >0 \qquad \forall \eps >0.
            \end{equation}
        \item If $K$ is compact, then the bound in \eqref{eq:CharTanSpace} is uniform in $a \in K$, that is, \label{lem:PropertiesMFLD(ii)}
            \begin{equation}
                \inf_{a \in K} \overline \delta_K(a,T_aK,\eps) > 0 \qquad \forall \eps >0.
            \end{equation}
        \item The (relative) interior points of $K$ (\ie those not lying on the boundary) are exactly those points $a \in K$, for which, in addition to \eqref{eq:CharTanSpace}, it holds \label{lem:PropertiesMFLD(iii)}
            \begin{equation}
                \underline \delta_K(a,T_aK,\eps) >0 \qquad \forall \eps > 0.
            \end{equation}
        \item The tangent bundle $TK:K \to \Gr(k,\RRN)$, interpreted as the map $a \mapsto T_aK$, is continuous. \label{lem:PropertiesMFLD(iv)}
    \end{enumerate}
\end{lemma}

The following lemma will be used several times throughout this and the following sections.
Recall that for $H \subset \RRmn$ closed and  $p \in (1,\infty)$, we denote by $C_{QR}(H,p) \in [1,\infty]$ the optimal constant in the quantitative rigidity estimate \eqref{eq:intro rig est}.
\begin{lemma} \label{lem:KornConstCont}
    Let $k \leq mn$.
    Then for every $p \in (1,\infty)$ the map
    \begin{equation}
        \Gr(k,\RRmn) \ni V \longmapsto \frac{1}{C_{QR}(V,p)} \in [0,1]
    \end{equation}
    is $2$-Lipschitz continuous with respect to the metric $d$ defined in \eqref{eq:DefMetricGrassmannian}.
    In particular, $C_{QR}:\Gr(k,\RRmn) \to [1,\infty]$ is continuous, and the set of $L^p$ quantitatively rigid subspaces is open in $\Gr(k,\RRmn)$.
\end{lemma}
\begin{proof}
Let $p \in (1,\infty)$ be fixed.
By homogeneity, we may write
\begin{equation}
    \frac{1}{C_{QR}(V)} \coloneq \frac{1}{C_{QR}(V,p)} =\inf_{\substack{u \in W^{1,p}(Q;\RRm)\\ \norm{Du - (Du)_Q}_{L^p}= 1}} \norm{\pi_{V^\perp }Du}_{L^p}.
\end{equation}
Now let $\eps >0$. Then we can find $u \in W^{1,p}(Q;\RRm)$ with $\norm{Du - (Du)_Q}_{L^p}= 1$ such that
\begin{equation} \label{eq:BdVperpDu}
    \norm{\pi_{V^\perp }Du}_{L^p} \leq \frac{1}{C_{QR}(V)} + \eps \leq 1 + \eps;
\end{equation}
without loss of generality we may additionally assume that $\pi_V((Du)_Q) = 0$.
For $W \in \Gr(k,\RRmn)$ it follows
\begin{equation}
    \begin{split}
    \frac{1}{C_{QR}(W)} & \leq  \norm{\pi_{W^\perp}Du}_{L^p} \leq \norm{\pi_{V^\perp}Du}_{L^p} + \norm{(\pi_{W^\perp} - \pi_{V^\perp})Du}_{L^p} \\
    &\leq                              \frac{1}{C_{QR}(V)} + \eps + \norm{(\pi_{W^\perp} - \pi_{V^\perp})Du}_{L^p}.
    \end{split}
\end{equation}
Further, since $\pi_V(Du)_Q = 0$ and by \eqref{eq:BdVperpDu}
\begin{equation}
    \norm{Du}_{L^p} \leq \norm{Du - (Du)_Q}_{L^p} +  \norm{\pi_{V^\perp}(Du)_Q}_{L^p} \leq 2 + \eps;
\end{equation}
so that by definition of the metric on $\Gr(k,\RRmn)$ it holds
\begin{equation}
    \norm{(\pi_{W^\perp} - \pi_{V^\perp})Du}_{L^p} \leq d(V,W)(2 + \eps).
\end{equation}
Hence,
\begin{equation}
    \frac{1}{C_{QR}(W)} \leq \frac{1}{C_{QR}(V)} + \eps +  d(V,W)(2 + \eps).
\end{equation}
As $V,W \in \Gr(k,\RRmn)$ and $\eps >0$ were arbitrary, this concludes the proof.
\end{proof}
Combining the preceding lemma with \autoref{lem:PropertiesMFLD}~\ref{lem:PropertiesMFLD(iv)}, we obtain the following.
\begin{corollary} \label{cor:Tangent space const cont}
    Let $K \subset \RRmn$ be a $k$-dimensional $C^1$-submanifold with boundary and $p \in (1,\infty)$.
    Then the map
    \begin{equation}
        K \ni A \longmapsto C_{QR}(T_AK,p) \in [1,\infty]
    \end{equation}
    is continuous.
\end{corollary}

\subsection{Proof of Necessity} \label{subsec:Necessity}
In this subsection we prove the necessity of the conditions \ref{intro(a)} and \ref{intro(b)} in \autoref{thm:GeneralRigidity}.
Let $p \in (1,\infty)$ and let $K \subset \RRmn$ be a $C^1$-submanifold with boundary satisfying $L^p$ quantitative rigidity.
Since $K$ is compact, this immediately implies that $K$ satisfies sequential rigidity, that is, \ref{intro(a)} holds.

Therefore, it only remains to show \ref{intro(b)}, that is, $T_AK$ satisfies exact rigidity for each $A \in K$.
We achieve this by proving directly that $T_AK$ satisfies $L^p$ quantitative rigidity using the following blow-up type argument.

\begin{lemma} \label{lem:QuantRigImpliesKorn}
    Let $K \subset \RRmn$ be a $C^1$-submanifold with boundary and $p \in (1,\infty)$.
    Then for every $A \in K$ it holds
    \begin{equation} \label{eq:QuantRigImpliesKorn}
        C_{QR}(T_AK,p) \leq C_{QR}(K,p).
    \end{equation}
\end{lemma}

\begin{proof}
Note that by \autoref{cor:Tangent space const cont}, it suffices to consider interior points of $K$.
Now assume that \eqref{eq:QuantRigImpliesKorn} fails for an interior point $A \in K$, and denote
$C_* \coloneq  C_{QR}(K,p)$ and $V \coloneq T_AK$.
Then there exist a map $u \in W^{1,p}(Q;\RRm)$ and $\kappa > 0$ such that
\begin{equation}
    1 = \norm{Du - (Du)_Q}_{L^p} \geq (C_*+\kappa) \norm{\dist_V(Du)}_{L^p}.
\end{equation}
Since $C^\infty(\overline Q;\RRm)$ is dense in $W^{1,p}(Q;\RRm)$, we may further assume that
\begin{equation}
    \norm{Du}_{L^\infty} \leq L < \infty.
\end{equation}
Let $\eps >0 $. By \autoref{lem:PropertiesMFLD}~\ref{lem:PropertiesMFLD(iii)}, it holds $\delta\coloneq  \underline \delta_K(A,V,\eps) > 0$, which means that
\begin{equation}
    \dist_K(F) \leq \dist_V(F-A) + \eps\abs{F-A} \qquad \forall F \in B(A,\delta).
\end{equation}
Now, for $\lambda>0$, define $u_\lambda (x)\coloneq  \lambda u(x) + Ax$.
Then $\norm{Du_\lambda - A}_{L^\infty} \leq \lambda L$, so if $\lambda L \leq \delta$, we can conclude
\begin{equation}
    \begin{split}
    \norm{\dist_K(Du_\lambda)}_{L^p} & \leq \norm{\dist_V(Du_\lambda-A)}_{L^p} + \eps \norm{Du_\lambda - A}_{L^p} \\
    & \leq \lambda \norm{\dist_V(Du)}_{L^p} + \eps \lambda L                  \\
    & \leq \lambda\left(\frac{1}{C_* + \kappa } + \eps L   \right).
    \end{split}
\end{equation}
Choosing $\eps$ sufficiently small, we still have
\begin{equation}
    \norm{\dist_K(Du_\lambda)}_{L^p} < \frac{\lambda}{C_*} = \frac{\norm{Du_\lambda - (Du_\lambda)_Q}_{L^p}}{C_*},
\end{equation}
contradicting the choice $C_* = C_{QR}(K,p)$.
\end{proof}
This concludes the proof of necessity.

\subsection{Proof of Sufficiency}
In this subsection we prove the sufficiency of the two conditions in \autoref{thm:GeneralRigidity}.
The main part of the proof is formulated as a separate and slightly more general result, \autoref{prop:TechnicalProp}, which will also be central in \autoref{sec:Stability}.
We begin by reformulating the rigidity estimate \eqref{eq:intro rig est}.
\begin{definition} \label{def:F-zeta}
    Let $K \subset \RRmn$ be closed and let $p \in (1, \infty)$.
    For $s > 0$, we define the \emph{class of admissible functions $\cF(s)$}
    to be all functions $u \in W^{1,p}(Q;\RRm)$
    for which the right-hand side of \eqref{eq:intro rig est} is bounded by $s$; that is,
    \begin{equation} \label{eq:def cS}
        \cF(s) \coloneq  \set[\big]{u \in W^{1,p}(Q;\RRm) \given \norm{\dist_K(Du)}_{L^p} \leq s}.
    \end{equation}
    We further define $\zeta(s)$ to be the supremum over the corresponding left-hand side,
    \begin{equation} \label{eq:def zeta}
        \zeta(s) \coloneq  \sup_{u \in \cF(s)} \norm{Du - (Du)_{Q}}_{L^p}.
    \end{equation}
\end{definition}
From this definition we obtain the following.
\begin{lemma} \label{lem:after def zeta}
    Let $K \subset \RRmn$ be closed, $p \in (1, \infty)$, and let $\zeta$ be given by \eqref{eq:def zeta}.
    \begin{enumerate}
        \item \label{lem:after def zeta(i)}
            Then $K$ satisfies $L^p$ quantitative rigidity if and only if there exists $C>0$ such that
            \begin{equation} \label{eq:zetaleCss}
                \zeta(s) \leq C s \qquad \forall s > 0.
            \end{equation}
        \item \label{lem:after def zeta(ii)} Assume in addition that $K$ is compact, and let $s_0 > 0$. Then \eqref{eq:zetaleCss} holds if there exists $C>0$ such that
            \begin{equation} \label{eq:zetaleCs}
                \zeta(s) \leq C s \qquad \forall s \in (0,s_0).
            \end{equation}
        \item \label{lem:after def zeta(iii)} $K$ satisfies sequential rigidity if and only if $\zeta(s) \to 0$ as $s \to 0$.
        \item \label{lem:after def zeta(iv)} The definition of $\zeta$ is \emph{scaling invariant} in the following sense:
            For any cube $\tilde Q \subset Q$ with sides parallel to the coordinate axes and any $u \in W^{1,p}(\tilde Q;\RRm)$, it holds
            \begin{equation} \label{eq:ScalingInvZeta}
                \left(\dashint_{\tilde Q} \abs{Du - (Du)_{\tilde Q}}^p \dx\right)^\frac{1}{p} \leq \zeta\left(\left[\dashint_{\tilde Q} \dist_K^p(Du) \dx\right]^\frac{1}{p}\right).
            \end{equation}
    \end{enumerate}
\end{lemma}
\begin{proof}
Item \ref{lem:after def zeta(i)} follows immediately from the definition.
To prove \ref{lem:after def zeta(ii)}, note that for $u \in \cF(s)$ it holds
\begin{equation} \label{eq:use compactness}
    \norm{Du - (Du)_Q}_{L^p} \leq 2\norm{\dist_K(Du)}_{L^p} + 2\diam(K) \leq 2s + 2\diam(K).
\end{equation}
If $s \geq s_0$, the right-hand side can be estimated by $(2 + 2\diam(K) s_0^{-1})s$, and the conclusion follows by taking the supremum over $u$.
Item \ref{lem:after def zeta(iii)} follows from \autoref{prop:seqRig}, and item \ref{lem:after def zeta(iv)} by a change of variables.
\end{proof}

Before stating the more technical \autoref{prop:TechnicalProp}, let us quickly outline the main ideas that are needed to prove \autoref{thm:GeneralRigidity}.
The general strategy is to linearize $\dist_K$ around some suitably chosen point $A \in K$, and to ultimately reduce the rigidity estimate for $K$ to a rigidity estimate for $T_AK$;
the latter follows from condition \ref{intro(b)}.
To this end, we split $Q$ into a superlevel set $S$ of $\abs{Du-(Du)_Q}$, and a sublevel set $Q \setminus S$.
On the sublevel set, $Du$ is sufficiently close to $A$ so that the remainder term from the linearization can be controlled;
the main difficulty lies in showing $\abs{Du-(Du)_Q}^p$ cannot concentrate too much on the superlevel set $S$.
To get some intuition for how this can be achieved, let us for now assume that we have a uniform bound of the form
\begin{equation} \label{eq:hyp unif bd}
    \norm{\dist_K(Du)}_{L^\infty(Q)} \leq s.
\end{equation}
This would directly result in a bound for the $BMO$-seminorm of $Du$: Indeed, by \eqref{eq:ScalingInvZeta} we would have
\begin{equation}
    \dashint_{\tilde Q} \abs{Du -(Du)_{\tilde Q}} \dx \leq \left(\dashint_{\tilde Q} \abs{Du -(Du)_{\tilde Q}}^p \dx\right)^\frac{1}{p} \leq \zeta(s)
\end{equation}
for all cubes $\tilde Q \subset Q$ with sides parallel to the coordinate axes.
By the John-Nirenberg inequality \cite{MR131498}, this would imply that the superlevel sets of $\abs{Du - (Du)_{Q}}$ decay exponentially fast:
\begin{equation} \label{eq:JNlemmaOrg}
    \abs[\big]{\set{x \in Q \given \abs{Du(x) - (Du)_Q}>t}} \leq C \exp\left(-\frac{t}{C \zeta(s)}\right).
\end{equation}
Since \ref{intro(a)} is equivalent to $\zeta(s) \to 0$ as $s \to 0$, such a bound would be enough to control the concentration of $\abs{Du-(Du)_Q}^p$.

In general, we cannot expect the bound \eqref{eq:hyp unif bd} to hold; instead we only know that $u \in \cF(s)$.
The idea is then to, in a sense, make this bound uniform by removing another small \emph{bad set} $E \subset Q$ where the maximal function of $\dist_K^p(Du)$ is large.
This will allow us to still obtain a bound of the type \eqref{eq:JNlemmaOrg} outside $E$ by proving a modified version of the John-Nirenberg lemma.

We now claim that \autoref{thm:GeneralRigidity} follows from the below proposition.
\begin{proposition} \label{prop:TechnicalProp}
    Let $p \in (1,\infty)$.
    For every $M>0$, there exist constants $\eps_0 = \eps_0(p,M) > 0$ and $L_0 = L_0(p,M) > 2$ such that the following holds.
    Let $K \subset \RRmn$ be compact and assume that there exists a family of subspaces
    \begin{equation} \label{eq:def V_A}
        (V_A)_{A \in K} \subset \bigcup_{k \leq mn}\Gr(k,\RRmn),
    \end{equation}
    such that
    \begin{equation} \label{eq:max Tan Korn leq M}
        \sup_{A \in K} C_{QR}(V_A,p) \leq M.
    \end{equation}
    If there exist $\delta>0$ and $s_0 > 0$ such that 
    \begin{equation} \label{eq:def delta}
        \inf_{A \in K} \overline \delta_K(A,V_A,\eps_0) \geq \delta
    \end{equation}
    and
    \begin{equation} \label{eq:cond on s_0}
        (1 + L_0)\zeta(s_0) + s_0 \leq \delta,
    \end{equation}
    then $\zeta$ as defined in \autoref{def:F-zeta} satisfies \eqref{eq:zetaleCs}, where $C = C(p,M,s_0, \diam(K))$. \medskip
\end{proposition}
\begin{proof}[Proof that \autoref{prop:TechnicalProp} implies \autoref{thm:GeneralRigidity}]
Let $K \subset \RRmn$ be a compact $k$-dimen\-sio\-nal $C^1$-submanifold with boundary satisfying conditions \ref{intro(a)} and \ref{intro(b)} of \autoref{thm:GeneralRigidity},
that is, $K$ is sequentially rigid and at each $A \in K$, the tangent space $T_AK$ satisfies exact rigidity.
Set
\begin{equation}
    V_A \coloneq  T_AK \in \Gr(k,\RRmn), \qquad A \in K.
\end{equation}
It then suffices to check that equations \eqref{eq:max Tan Korn leq M}, \eqref{eq:def delta} and \eqref{eq:cond on s_0} are satisfied.

By \ref{intro(b)} and \autoref{cor:Subspace-QuantIffExact}, it follows that $T_AK$ satisfies $L^p$ quantitative rigidity for each $A \in K$,
or equivalently,
\begin{equation}
    C_{QR}(T_AK,p) < \infty \qquad \forall A \in K.
\end{equation}
By \autoref{cor:Tangent space const cont} and since $K$ is compact, the map
$A \mapsto C_{QR}(T_AK,p)$ attains a maximum $M \in [1,\infty)$ on $K$, so that \eqref{eq:max Tan Korn leq M} is satisfied.
Further, by \autoref{lem:PropertiesMFLD}~\ref{lem:PropertiesMFLD(ii)}, it follows that the constant $\delta$ in
\eqref{eq:def delta} can be chosen strictly positive.
Finally, by \autoref{prop:seqRig}, sequential rigidity of $K$ is equivalent to
$\zeta(s) \to 0$ as $s \to 0$, so that for $s_0>0$ sufficiently small, \eqref{eq:cond on s_0} is satisfied as well.
\end{proof}

\begin{proof}[Proof of \autoref{prop:TechnicalProp}]
Assume that $K \subset \RRmn$ and $(V_A)_{A \in K} \subset \bigcup_{k \leq mn}\Gr(k,\RRmn)$ satisfy the hypotheses of \autoref{prop:TechnicalProp},
where the constants $\eps_0 = \eps_0(p,M) \in (0,1]$ and $L_0 = L_0(p,M) > 2$ will be specified at the end of the proof.
Recall that by \eqref{eq:def delta} and \autoref{def:delta_K}, for $A \in K$ and $F \in \RRmn$ it holds
\begin{equation} \label{eq:WhatDoesDeltaDo}
    \dist_{V_A}(F-A) \leq \dist_K(F) + \eps_0 \abs{F-A} \qquad  \text{if } \abs{F-A} \leq \delta.
\end{equation}
Now, let $s \in (0,s_0)$ and fix $u \in \cF(s)$. Choose $A \in K$ such that
\begin{equation}
    \dist_K((Du)_{Q}) = \abs{(Du)_{Q} - A}.
\end{equation}
Note that by the triangle inequality
\begin{equation} \label{eq:DuQ0minusA0}
    \abs{(Du)_{Q} - A} \leq \norm{Du - (Du)_{Q}}_{L^p} + \norm{\dist_K(Du)}_{L^p} \leq \zeta(s) + s.
\end{equation}
Consider the superlevel set
\begin{equation} \label{eq:def S}
    S\coloneq  \set{x \in Q \given \abs{Du(x) - (Du)_{Q}} > L_0 \zeta(s_0) },
\end{equation}
outside which we are able to use the linearization inequality \eqref{eq:WhatDoesDeltaDo}.
Indeed, by \eqref{eq:cond on s_0} and \eqref{eq:DuQ0minusA0} it holds for all $x \in Q \setminus S$
\begin{equation}
    \abs{Du(x) - A} \leq \abs{Du(x) - (Du)_{Q}} + \abs{(Du)_{Q} - A}  \leq (1 + L_0)\zeta(s_0) + s \leq \delta,
\end{equation}
where we also used that $\zeta$ is nondecreasing by definition.
Hence, first using the rigidity estimate for $V_A$ and \eqref{eq:max Tan Korn leq M}, then \eqref{eq:WhatDoesDeltaDo} on $Q\setminus S$, it follows
\begin{equation} \label{eq:Korn + linearize}
    \begin{split}
    \norm{Du - (Du)_{Q}}_{L^p(Q)} & \leq M \norm{\dist_{V_A}(Du-A)}_{L^p(Q)}                                            \\
    & \leq                                    M\big(s + \eps_0 \norm{Du - A}_{L^p(Q \setminus S)} +  \norm{Du - A}_{L^p(S)}\big).
    \end{split}
\end{equation}
The idea is to absorb all terms which do not depend linearly on $s$ into the left-hand side;
to this end we need to replace $A$ on the right side again by $(Du)_Q$.
First, we deal with the integral over $S$: Note that by Markov's inequality it holds
\begin{equation} \label{eq:BdonS}
    \abs{S} \leq \left(\frac{\norm{Du - (Du)_{Q}}_{L^p(Q)}}{L_0 \zeta(s_0)}\right)^p \leq \left(\frac{\zeta(s)}{L_0 \zeta(s_0)}\right)^p\leq L_0^{-p},
\end{equation}
so that by \eqref{eq:DuQ0minusA0} (and also using $L_0^{-1}\leq 1$) we can estimate
\begin{equation} \label{eq:Du - A Lp(S)}
    \begin{split}
    \norm{Du - A}_{L^p(S)} \leq \norm{Du - (Du)_{Q}}_{L^p(S)} + L_0^{-1}\zeta(s) + s.
    \end{split}
\end{equation}
Similarly, it follows
\begin{equation} \label{eq:Du - A Lp(Q-S)}
    \eps_0\norm{ Du - A}_{L^p(Q \setminus S)} \leq \eps_0 \bigl(\norm{Du - (Du)_{Q}}_{L^p(Q\setminus S)} + \zeta(s) + s\bigr) \leq 2\eps_0\zeta(s) + s.
\end{equation}
Hence, by \eqref{eq:Korn + linearize}, \eqref{eq:Du - A Lp(S)} and \eqref{eq:Du - A Lp(Q-S)}
\begin{equation} \label{eq:GenRigIntermediateEst}
    \norm{Du - (Du)_{Q}}_{L^p(Q)} \leq M\Bigl(3s + (2\eps_0 + L_0^{-1})\zeta(s) + \norm{Du - (Du)_{Q}}_{L^p(S)}\Bigr).
\end{equation}
While the term $3Ms$ already has the desired form and the term involving $\zeta(s)$ can be absorbed into the left-hand side, the last summand however seems more problematic:
We integrate the same quantity as on the left-hand side, except that it is also multiplied by $M$;
the only leverage we have is that we integrate only over $S$.
It is clear that the estimate \eqref{eq:BdonS} for the size of $S$ given by Markov's inequality is not enough to estimate this term,
since we additionally need to control the growth of $\abs{Du -(Du)_{Q}}$.
As noted before, we achieve this by proving a John-Nirenberg type inequality outside some bad set $E \subset Q$ where the maximal function of $\dist_K^p(Du)$ is large.

To this end, denote
\begin{equation}
    \cQ \coloneq  \set{2^{-k}([0,1]^n + z) \subset Q \given k \in \NN_{\geq 0} \text{ and } \; z \in \ZZ^n}
\end{equation}
the set of closed dyadic subcubes of $Q$ and
consider the corresponding dyadic Hardy-Littlewood maximal function
\begin{equation}
    M_d[\dist_K^p(Du)](x) \coloneq  \sup_{\set{\tilde Q \in \cQ \given x \in \tilde Q}} \dashint_{\tilde Q} \dist_K^p(Du) \dz.
\end{equation}
Now consider the \emph{bad set}
\begin{equation} \label{eq:DefErrorSet}
    E\coloneq  \set[\big]{x \in Q \given M_d[\dist_K^p(Du)](x) > s_0^p} = \bigcup\set[\bigg]{\tilde Q \in \cQ \given \dashint_{\tilde Q} \dist_K^p(Du) \dz > s_0^p}.
\end{equation}
Since every cover of dyadic cubes admits a disjoint subcover, we may estimate
\begin{equation}
    \abs{E} \leq \left(\frac{\norm{\dist_K(Du)}_{L^p(Q)}}{s_0}\right)^p \leq  \left(\frac{s}{s_0}\right)^p.
\end{equation}
Similar to \eqref{eq:use compactness}, it follows
\begin{equation} \label{eq:RemoveErrorSet}
    \norm{Du - (Du)_{Q}}_{L^p(S)} \leq 2s + \frac{2\diam(K)}{s_0}s + \norm{Du - (Du)_{Q}}_{L^p(S \setminus E)}.
\end{equation}
Consider now the modified superlevel sets
\begin{equation} \label{eq:DefTildeS}
    \tilde S_t \coloneq  \set{x \in Q \setminus E \given \abs{Du(x) - (Du)_{Q}}> t}, \qquad t > 0,
\end{equation}
and note that
\begin{equation} \label{eq:S_setminus_E}
    S\setminus E = \tilde S_{L_0\zeta(s_0)}.
\end{equation}
\autoref{prop:TechnicalProp} now follows from the following John-Nirenberg type estimate, whose proof is postponed to the end of the subsection.
\begin{lemma} \label{lem:BMOtypeEst}
    There exists a constant $C_1=C_1(p)>0$ such that the following holds.
    Let $s_0 > 0$, $s \in (0,s_0)$, $u \in \cF(s)$ and let $\tilde S_t$ be given by \eqref{eq:DefTildeS}.
    Then, for $t> C_1\zeta(s_0)$ it holds
    \begin{equation}
        \abs{\tilde S_t }\leq \left(\frac{\zeta(s)}{\zeta(s_0)}\right)^p C_1\exp\left(- \frac{t }{C_1 \zeta(s_0)} \right).
    \end{equation}
\end{lemma}
Indeed, assume that \autoref{lem:BMOtypeEst} holds.
By \eqref{eq:S_setminus_E} and Fubini, we can rewrite the last remaining term in \eqref{eq:RemoveErrorSet} in terms of $\tilde S_t$.
\begin{equation}
    \begin{split}
    \norm{Du - (Du)_{Q}}_{L^p(S \setminus E)}^p &= \int_{\tilde S_{\scriptscriptstyle L_0\zeta(s_0)}}\abs{Du - (Du)_{Q}}^p \dx\\
    &= \int_{0}^\infty pt^{p-1}  \abs[\big]{\tilde S_{\max\set{t , L_0 \zeta(s_0)}}} \dt.
    \end{split}
\end{equation}
Provided that $L_0 \geq C_1$, we can insert the estimate from \autoref{lem:BMOtypeEst}:
\begin{equation}
    \begin{split}
    \int_{0}^\infty pt^{p-1}  \abs[\big]{\tilde S_{\max\set{t , L_0 \zeta(s_0)}}} \dt & \leq  \left(\frac{\zeta(s)}{\zeta(s_0)}\right)^p C_1\int_{L_0 \zeta(s_0)}^\infty pt^{p-1}  \exp\left(- \frac{t}{C_1 \zeta(s_0)} \right) \dt \\
    & \quad + \left(\frac{\zeta(s)}{\zeta(s_0)}\right)^p C_1(L_0 \zeta(s_0))^p \exp\left(- \frac{L_0 \zeta(s_0) }{C_1 \zeta(s_0)} \right)     \\
    &= \zeta(s)^pC_1\left( \int_{L_0}^\infty pt^{p-1}  e^{- \frac{t}{C_1}} \dt + (L_0)^p e^{-\frac{L_0}{C_1}} \right)  \\
    & \leq \left(\zeta(s)C_2 e^{-\frac{L_0}{C_2}}\right)^p
    \end{split}
\end{equation}
for some constant $C_2=C_2(p)>0$.
It is vital that $C_2$ can be chosen independently of $L_0$ and $s_0$, since we want to choose $L_0$ in a way that makes the right-hand side small.
Collecting all the terms from \eqref{eq:GenRigIntermediateEst}, \eqref{eq:RemoveErrorSet} and the previous equation,
we arrive at
\begin{equation}
    \begin{split}
    \zeta(s)  &= \sup_{u \in \cF(s)}\norm{Du - (Du)_{Q}}_{L^p(Q)}                                                        \\
    &\leq          M \Bigl(2\eps_0 + L_0^{-1} + C_2e^{-\frac{L_0}{C_2}}\Bigr) \zeta(s) + M\biggl(5 + \frac{2\diam(K)}{s_0}\biggr) s.
    \end{split}
\end{equation}
Now, choose $\eps_0=\eps_0(p,M)>0$ and $L_0 = L_0(p,M) > C_1$ so that the first term on the right-hand side is bounded by $\frac{1}{2}\zeta(s)$,
and absorb it into the left-hand side.
Up to proving \autoref{lem:BMOtypeEst}, this concludes the proof of \autoref{prop:TechnicalProp}.
\end{proof}

The proof of \autoref{lem:BMOtypeEst} is an adaptation of the proof of the John-Nirenberg inequality \cite{MR131498};
in particular it relies heavily on the Calderón-Zygmund decomposition (see also \cite[Ch.~I, Thm.~4]{Stein}).
\begin{lemma}[Calderón-Zygmund decomposition, {\cite{MR131498}}]
    Let $\tilde Q \in \cQ$ be a dyadic cube and $f \in L^1(\tilde Q)$ a nonnegative function.
    Let $t>0$ be a parameter such that
    \begin{equation}
        t > \dashint_{\tilde Q} f \dx.
    \end{equation}
    Then there exists a family ${\set{Q_k}}_k \subset \cQ$ of disjoint, dyadic subcubes of $\tilde Q$ such that
    \begin{enumerate}
        \item $f(x) \leq t\;$ for a.e. $x \in \tilde Q \setminus \bigcup_{k}Q_k$, \label{CZD:item(a)}
        \item $ t < \dashint_{Q_k} f \dx \leq 2^n t\;$ for all $k$, \label{CZD:item(b)}
        \item $ \abs[\big]{\bigcup \nolimits_k Q_k} \leq \frac{1 }{t}\int_{\tilde Q} f \dx $. \label{CZD:item(c)}
    \end{enumerate}
\end{lemma}

\begin{proof}[Proof of \autoref{lem:BMOtypeEst}]
We proceed in three steps.

\Step{1}
Applying the Calderón-Zygmund decomposition to $\abs{Du - (Du)_{Q}}^p \in L^1(Q)$ with parameter
\begin{equation}
    t := (2 \zeta(s_0))^p > \zeta(s)^p \geq \int_{Q} \abs{Du - (Du)_{Q}}^p \dx,
\end{equation}
we obtain a family ${\set{Q_1^k}}_{k} \subset \cQ$ of pairwise disjoint dyadic subcubes of $Q$
satisfying assertions \ref{CZD:item(a)} to \ref{CZD:item(c)} of the Calderón-Zygmund decomposition.
Without explicitly renaming the family, let us discard all cubes $Q_1^k$ for which $Q_1^k \subset E$, and set
\begin{equation}
    R_0 \coloneq  Q \setminus E, \qquad R_1 \coloneq   \bigcup_{k}Q_1^{ k} .
\end{equation}
Then assertions \ref{CZD:item(a)} to \ref{CZD:item(c)} become
\begin{subequations}
    \begin{align}
        \abs{Du(x) - (Du)_{Q}} &\leq 2 \zeta(s_0) \qquad \text{for a.e. } x \in R_0 \setminus R_1, \noeqref{eq:CZ1}\label{eq:CZ1}\\
        (2 \zeta(s_0))^p &< \dashint_{Q_1^k}\abs{Du - (Du)_{Q}}^p \dx \leq 2^{n}(2\zeta(s_0))^p \qquad \forall k, \label{eq:CZ2}\\
        \abs{R_1} &\leq\left(\frac{\zeta(s)}{2\zeta(s_0)}\right)^p. \noeqref{eq:CZ3}\label{eq:CZ3}
    \end{align}
\end{subequations}
Denote
\begin{equation}
    \lambda \coloneq 2^{\frac{n}{p}+1},
\end{equation}
then from Jensen's inequality and \eqref{eq:CZ2} we further conclude
\begin{equation} \label{eq:BMOestControlAvg}
    \abs{(Du)_{Q} - (Du)_{Q_1^k}} \leq \bigg(\dashint_{Q_1^k}\abs{Du - (Du)_{Q}}^p \dx \bigg)^\frac{1}{p}\leq \lambda \zeta(s_0).
\end{equation}

\Step{2} We continue by induction.
For $j \in \NN$, consider the following statement:\\
\emph{There exists a family ${\set{Q_j^k}}_{k} \subset \cQ$ of disjoint dyadic cubes
    so that for $R_j \coloneq  \bigcup_{k}Q_j^k$ it holds
    \begin{subequations}
        \begin{align}
            \abs{Du(x) - (Du)_{Q}} &\leq j\lambda  \zeta(s_0) \qquad \text{for a.e. } x \in R_0 \setminus R_j; \label{eq:CZ4}\\
            \abs{R_j} &\leq2^{-jp}\left(\frac{\zeta(s)}{\zeta(s_0)}\right)^p; \label{eq:CZ5}
        \end{align}
        and for each $k$, we have $Q_j^k \nsubseteq E$ and
        \begin{equation}
            \abs{(Du)_{Q} - (Du)_{Q_j^k}} \leq j\lambda \zeta(s_0). \label{eq:CZ6}
        \end{equation}
\end{subequations} }
In Step 1 we have shown that the statement holds for $j=1$.
Now assume that the statement holds for some $j \geq 1$; we claim that it is also true for $j+1$.
Consider one of the cubes $Q_j^k$.
By assumption, $Q_j^k \nsubseteq E$, so by definition of the bad set, \eqref{eq:DefErrorSet}, it holds
\begin{equation}
    \dashint_{Q_j^k} \dist_K^p(Du) \dx \leq s_0^p,
\end{equation}
and by the scaling invariant estimate \eqref{eq:ScalingInvZeta} it follows
\begin{equation}
    \dashint_{Q_j^k} \abs{Du - (Du)_{Q_j^k}}^p \dx \leq \zeta(s_0)^p.
\end{equation}
This allows us to apply the Calderón-Zygmund decomposition on the cube $Q^k_j$ to the function $\abs{Du - (Du)_{Q_j^k}}^p$ with the \emph{same} parameter $(2\zeta(s_0))^p$ as in Step 1,
resulting in a new family ${\set{Q_{j+1}^{k,l}}}_l \subset \cQ$ of disjoint dyadic subcubes of $Q_j^k$ which satisfies
\begin{subequations}
    \begin{align}
        \abs{Du - (Du)_{Q_j^k}} &\leq 2 \zeta(s_0) \qquad \text{for a.e. } x \in Q_j^k \setminus \bigcup_l Q_{j+1}^{k,l}, \noeqref{eq:CZ1}\label{eq:CZ1 p}\\
        (2 \zeta(s_0))^p &< \dashint_{Q_{j+1}^{k,l}}\abs{Du - (Du)_{Q_j^k}}^p \dx \leq 2^{n}(2\zeta(s_0))^p \qquad \forall l, \label{eq:CZ2 p}\\
        \abs[\bigg]{\bigcup_l Q_{j+1}^{k,l}} &\leq2^{-p}\abs{Q_j^k}. \noeqref{eq:CZ3}\label{eq:CZ3 p}
    \end{align}
\end{subequations}
Similarly to \eqref{eq:BMOestControlAvg}, from \eqref{eq:CZ2 p} it further follows
\begin{equation} \label{eq:BMOestControlAvg p}
    \abs{(Du)_{Q_j^k} - (Du)_{Q_j^{k,l}}}\leq \lambda \zeta(s_0).
\end{equation}
Without explicitly changing the notation, let us again discard all cubes $Q_j^{k,l}$ which are contained in the bad set~$E$.
Now repeat this for each $k$; and let
\begin{equation}
    {\set{Q_{j+1}^{k'}}}_{k'} := \bigcup_k{\set{Q_{j+1}^{k,l}}}_{l} \subset \cQ
\end{equation}
denote the union of the resulting families.
Further, let $R_{j+1} \coloneq \bigcup_{k'} Q_{j+1}^{k'}$.
For each $k$, from \eqref{eq:CZ1 p} and \eqref{eq:CZ6}, we obtain
\begin{equation}
    \abs{Du(x) - (Du)_{Q}} \leq (j+1)\lambda  \zeta(s_0) \qquad \text{for a.e. } x \in Q_j^k \setminus \bigg(E \cup{ \bigcup_l Q_{j+1}^{k,l}}\bigg). \label{eq:CZ7}
\end{equation}
By the induction hypothesis \eqref{eq:CZ4}, the same bound also holds on $R_0 \setminus R_j = R_0 \setminus \bigcup_{k}Q_j^k$.
Thus, it holds on all of $R_0 \setminus R_{j+1}$, and \eqref{eq:CZ4} is satisfied for $j+1$ and the family ${\set{Q_{j+1}^{k'}}}_{k'}$.

Further, summing \eqref{eq:CZ3 p} over $k$, and applying the induction hypothesis~\eqref{eq:CZ5}, we infer
\begin{equation}
    \abs{R_{j+1}} = \sum_{k,l} \abs{Q_{j+1}^{k,l}} \leq 2^{-p} \sum_k \abs{Q_j^k} = 2^{-p}\abs{R_j} \leq 2^{-(j+1)p}\left(\frac{\zeta(s)}{\zeta(s_0)}\right)^p,
\end{equation}
so that \eqref{eq:CZ5} is satisfied for $j+1$ as well.
Finally, for each $k, l$ we obtain from \eqref{eq:CZ6} (for $j$) and \eqref{eq:BMOestControlAvg p} that
\begin{equation}
    \abs{(Du)_{Q} - (Du)_{Q_{j+1}^{k,l}}} \leq \abs{(Du)_{Q} - (Du)_{Q_{j}^{k}}} + \abs{(Du)_{Q_j^k} - (Du)_{Q_{j+1}^{k,l}}} \leq (j+1)\lambda  \zeta(s_0).
\end{equation}
This shows that \eqref{eq:CZ6} holds for $j+1$, finishing the proof of the induction step.

\Step{3} Hence, the statement holds for all $j \in \NN$.
Now, for $t > \lambda  \zeta(s_0)$, let $j \in \NN$ such that
\begin{equation}
    j\lambda\zeta(s_0) < t \leq (j+1) \lambda\zeta(s_0).
\end{equation}
Hence, by \eqref{eq:CZ4} and \eqref{eq:CZ5},
\begin{align}
    \abs{\tilde S_t}&=\abs{\set{x \in R_0 \given \abs{Du(x) - (Du)_{Q}}> t}} \leq\abs{R_j} \\
    &\leq \left(\frac{\zeta(s)}{\zeta(s_0)}\right)^p \exp\left(-j p \log 2 \right)
    \leq \left(\frac{\zeta(s)}{\zeta(s_0)}\right)^p 2^p\exp\left(- t\frac{p\log 2 }{\lambda \zeta(s_0)} \right).
\end{align}
This concludes the proof of \autoref{lem:BMOtypeEst}.
\end{proof}

We end the section with the following simple but important corollary of the proof of \autoref{thm:GeneralRigidity}.
More precisely, it follows from the observation that if the superlevel set $S$ introduced in \eqref{eq:def S} is empty to begin with, 
then there is no need to make any assumptions on $\zeta$ such as \eqref{eq:cond on s_0}.
\begin{corollary} \label{cor:smallOsc}
    Let $K \subset \RRmn$ be a compact $C^1$-submanifold with boundary and $p \in (1, \infty)$.
    Assume that $K$ satisfies \ref{intro(b)}, that is, for every $A \in K$, the space $T_AK$ satisfies exact rigidity.
    Then there exists a constant $\delta=\delta(K,p) > 0$ such that for every
    $u \in W^{1,\infty}(Q;\RRm)$ satisfying
    \begin{equation}
        \osc_Q Du  \coloneq \diam \supp Du_\#(\cL^n \mres Q) \leq {\delta},
    \end{equation}
    it holds
    \begin{equation}
        \norm{Du - (Du)_Q}_{L^p} \leq C(K,p) \norm{\dist_K(Du)}_{L^p}.
    \end{equation}
\end{corollary}

\section{Proof of {\autoref*{thm:Stability}}} \label{sec:Stability}
The proof will mainly rely on \autoref{prop:TechnicalProp}.
The first step of the proof is showing that $\overline \delta$, defined in \autoref{def:delta_K}, is stable under small graphical perturbations.

\begin{lemma} \label{lem:computeDelta_K}
    Let $K \subset \RRmn$ be closed and
    let $\Phi:\RRmn \to \RRmn$ be a bi-Lipschitz map satisfying \eqref{eq:LipCondonPhi} for some $\rho \in (0,\frac{1}{2}]$.
    Set $\tilde K \coloneq \Phi(K)$
    and let $V \in \Gr(k,\RRmn)$ be a subspace.
    Then for all $A \in K$ and $\eps > 0$, it holds
    \begin{equation} \label{eq:computeDelta_K}
        \overline \delta_{\tilde K}(\Phi(A),V,2\eps + 8 \rho) \geq \frac{1}{2}\overline\delta_{K}(A,V,\eps).
    \end{equation}
\end{lemma}
\begin{proof}[Proof of \autoref{lem:computeDelta_K}]
By \eqref{eq:LipCondonPhi} we have
\begin{equation} \label{eq:FminusG-prime}
    (1 - \rho)\abs{F - G} \leq \abs{\Phi(F) - \Phi(G)} \leq (1 + \rho)\abs{F - G}\qquad \forall F, G \in \RRmn
\end{equation}
and
\begin{equation} \label{eq:dist_VPhi}
    \dist_V(\Phi(F) - \Phi(G)) \leq \dist_V(F-G) + \rho \abs{F - G} \qquad \forall F, G \in \RRmn.
\end{equation}
Furthermore, write
\begin{equation} \label{eq:distKprime}
    \dist_{\tilde K}(\Phi(F)) = \inf_{B \in K} \abs{\Phi(F) - \Phi(B)}.
\end{equation}
Note that in the infimum we may disregard all matrices $B$  for which  $\abs{F - B} > 3\abs{F - A}$,
since by \eqref{eq:FminusG-prime} and $\rho \leq \frac{1}{2}$, this would already imply
\begin{equation}
    \abs{\Phi(F) - \Phi(B)} > \abs{\Phi(F) - \Phi(A)} \geq \dist_{\tilde K}(\Phi(F)).
\end{equation}
Hence, by \eqref{eq:FminusG-prime},
\begin{equation} \label{eq:dist_tildeKPhi(F)}
    \dist_{\tilde K}(\Phi(F)) \geq \inf_{\substack{B \in K\\\abs{F - B} \leq3\abs{F - A}}} (1-\rho)\abs{F - B} \geq \dist_K(F) - 3\rho\abs{F - A} .
\end{equation}
Now let $\eps > 0$ and $\delta \coloneq  \overline \delta_{K}(A,V,\eps)$.
If $\delta = 0$, we are done.
Otherwise, let
$ G \in B\big(\Phi(A),\delta/2\big)$
be arbitrary and set $F\coloneq \Phi^{-1}(G)$.
Since $\rho \leq \frac{1}{2}$, we still have $ F \in B(A,\delta)$,
so that by definition of $\delta$
\begin{equation} \label{eq:defofdelta}
    \dist_V(F-A) - \dist_K(F) \leq \eps \abs{F - A}.
\end{equation}
Applying \eqref{eq:dist_VPhi} and \eqref{eq:dist_tildeKPhi(F)}, \eqref{eq:defofdelta}, and then \eqref{eq:FminusG-prime}  it follows
\begin{align}
    \dist_{V}\big(\Phi(F) - \Phi(A)\big) - \dist_{\tilde K}(\Phi(F))
    &\leq \dist_V(F-A) - \dist_K(F) + 4\rho \abs{F - A} \\
    &\leq (\eps + 4 \rho) \abs{F - A} \\
    &\leq (2\eps + 8 \rho) \abs{\Phi(F) - \Phi(A)}.
\end{align}
As $G = \Phi(F) \in B\big(\Phi(A),\delta/2\big)$ was arbitrary, this concludes the proof. \medskip
\end{proof}

\begin{proof}[Proof of {\autoref{thm:Stability}}]
Let $K \subset \RRmn$ be a compact $C^1$-submanifold with boundary satisfying $L^p$ quantitative rigidity.
Assume further that $\Phi:\RRmn \to \RRmn$ satisfies \eqref{eq:LipCondonPhi}, where $\rho = \rho(K,p)>0$ will be specified later. Denote $\tilde K \coloneq \Phi(K)$.
Note that the rigidity constant $C_{QR}(\tilde K,p)$ is invariant under translations of $\tilde K$, so that without loss of generality we may assume $\Phi(A)=A$ for some fixed $A \in K$.
Then \eqref{eq:LipCondonPhi} implies
\begin{equation} \label{eq:distPhiK-K}
    \dist_{\tilde K}(F) \leq \dist_K(F) + \norm{\Phi - \id}_{L^\infty(B(A,\diam(K)))} \leq \dist_K(F) + \rho \diam(K).
\end{equation}
For $s>0$, define $\tilde \cF(s)$ and $\tilde \zeta(s)$ in analogy to \autoref{def:F-zeta}, that is,
\begin{equation} \label{eq:tilde F-zeta}
    \begin{aligned}
        \tilde \cF(s) &\coloneq  \set[\big]{u \in W^{1,p}(Q;\RRm) \given \norm{\dist_{\tilde K}(Du)}_{L^p} \leq s}, \\
        \tilde \zeta(s) &\coloneq  \sup_{u \in \tilde \cF(s)} \norm{Du - (Du)_{Q}}_{L^p}.
    \end{aligned}
\end{equation}
By \autoref{lem:after def zeta}, it suffices to show that there exist constants $s_0 = s_0(K,p) >0$ and $C=C(K,p)>0$ such that
\begin{equation} \label{eq:zetaPrimeLeqCs}
    \tilde \zeta(s) \leq Cs \qquad \forall s \in (0,s_0).
\end{equation}
In order to apply \autoref{prop:TechnicalProp}, we first need to specify a family of subspaces $(W_{B})_{B \in \tilde K}$.
If $\Phi$ were a $C^1$-diffeomorphism, the intuitive choice would be to set $W_{B} \coloneq T_{B}\tilde K$.
However, if $\Phi$ is merely bi-Lipschitz continuous, $\tilde K$ does not admit well-defined tangent spaces everywhere, so we might as well set
\begin{equation}
    W_{\Phi(A)} \coloneq  V_A = T_AK, \qquad A \in K.
\end{equation}
This also has the advantage that by \autoref{lem:QuantRigImpliesKorn} it automatically holds
\begin{equation}
    \max_{B \in \tilde K} C_{QR}(W_{B},p) = \max_{A \in K} C_{QR}(T_AK,p) \leq C_{QR}(K,p) \eqcolon M < \infty.
\end{equation}
Applying \autoref{lem:after def zeta} with $K$ replaced by $\tilde K$ and $\zeta$ replaced by $\tilde \zeta$, 
we find constants $\eps_0(p,M)>0$ and $L_0(p,M)>0$ so that \eqref{eq:zetaPrimeLeqCs} (and thus also \autoref{thm:Stability}) 
follow if we show that
\begin{equation} \label{eq:def delta prime}
    \inf_{B \in \tilde K} \overline \delta_{\tilde K}(B,W_{B},\eps_0) \geq \delta> 0
\end{equation}
and
\begin{equation} \label{eq:cond on s_0 prime}
    (1 + L_0)\tilde \zeta(s_0) + s_0 \leq \delta,
\end{equation}
where $\delta > 0$ and $s_0 >0$ depend \emph{only} on $K$ and $p$.
We first deal with \eqref{eq:def delta prime},
which is a consequence of  \autoref{lem:computeDelta_K}.
Indeed, if we add the condition
\begin{equation}
    \rho \leq \min\set[\Big]{\frac{\eps_0}{16}, \frac{1}{2}},
\end{equation}
then by \autoref{lem:computeDelta_K} it follows
\begin{equation}
    \inf_{B \in \tilde K} \overline \delta_{\tilde K}(B,W_{B},\eps_0) \geq \frac{1}{2}\inf_{A \in K} \overline \delta_K(A,T_AK,\eps_0 / 4) \eqcolon  \delta.
\end{equation}
Since $K$ is a compact $C^1$-submanifold with boundary, by \autoref{lem:PropertiesMFLD}~\ref{lem:PropertiesMFLD(ii)} it follows that $\delta$ is strictly positive;
moreover $\delta$ depends only on $K$ and on $\eps_0 = \eps_0(p,M) = \eps_0(p,K)$.

Finally, \eqref{eq:cond on s_0 prime} follows from the assumption that $K$ satisfies $L^p$ quantitative rigidity.
Indeed, recalling \eqref{eq:distPhiK-K} and \eqref{eq:tilde F-zeta}, we have for all $s>0$
\begin{equation}
    \tilde \zeta(s) \leq \zeta(s + \rho \diam(K)) \leq C(K,p)(s + \rho \diam(K)).
\end{equation}
Thus, by choosing $s_0$ and $\rho$ sufficiently small, we can always guarantee that \eqref{eq:cond on s_0 prime} is satisfied.
\end{proof}

\section{Examples} \label{sec:Examples}

We begin by considering arguably the simplest setting, namely the \emph{$N$-gradient problem}:
Consider the set $K=\set{A_1,\dots,A_N} \subset \RRmn$, where $N \geq 2$.
Already this example reveals the subtlety of the problem.

As we have seen earlier, in order for $K$ to satisfy exact rigidity, it is necessary that
$K$ contains no rank-one connections, that is $\rank(A_i - A_j) > 1$ for $i \neq j$.
Furthermore, since $K$ is a compact $0$-manifold and thus trivially satisfies condition \ref{intro(b)} of \autoref{thm:GeneralRigidity}, it follows that
$K$ satisfies $L^p$ quantitative rigidity if and only if it satisfies sequential rigidity.
Alternatively, one can argue as follows: If $K$ satisfies sequential rigidity, then the sets $\set{A_1},\dots,\set{A_N}$ are \emph{incompatible} in the sense of \autoref{ex:incompatible} below;
in this particular situation $L^p$ quantitative rigidity then already follows from the estimate \eqref{eq:DeLellisSzekelyhidi} of {\sc De Lellis and Székelyhidi} \cite{MR2253059}.
\begin{example}[The two-gradient problem]
    Let $K = \set{A_1, A_2} \subset \RRmn$, where $\rank(A_1-A_2)>1$.
    Then $K$ is sequentially rigid, a result first obtained by {\sc Ball and James} \cite{MR906132}.
\end{example}
\begin{example}[The three-gradient problem]
    Assume that $K = \set{A_1, A_2, A_3} \subset \RRmn$ contains no rank-one connections.
    Then $K$ is sequentially rigid. By \autoref{prop:seqRig}, this is equivalent to $K$ satisfying exact rigidity and every homogeneous gradient Young measure supported in $K$ being trivial.
    For exact rigidity see \cite[Thm.~2.5]{MR1731640};
    the latter was shown by {\sc \v{S}verák} \cite{MR1179688}.
\end{example}
\begin{example}[The four-gradient problem] \label{ex:TartarSquare}
    Let $K = \set{A_1,A_2,A_3,A_4 } \subset\RR^{2\times2}$,
    where $A_1 = \diag(-1,-3) = - A_3$ and $A_2= \diag(-3,1)=-A_4$.
    Then $K$ contains no rank-one connections, but does not satisfy sequential rigidity.
    In fact, one can show that $\set{\diag(s,t) \in \RR^{2\times 2} \given s,t \in [-1,1]} \subset K^{qc}$.
    This example is also known as the \emph{Tartar square} and was found independently by several authors, see the discussion in \cite[Sect.~2.5]{MR1731640}.

    Nevertheless, {\sc Chlebík and Kirchheim} \cite{MR1932170} proved that \emph{any} set of the form $K = \set{A_1,A_2,A_3,A_4 } \subset\RRmn$
    containing no rank-one connections still satisfies exact rigidity.
\end{example}
\begin{example}[The 5-gradient problem]
    As demonstrated by {\sc Kirchheim and Preiss} \cite{MR1817378, Kirchheim}, there is a set
    $K = \set{A_1,\dots,A_5} \subset \RR^{2\times2}_\sym$ containing no rank-one connections that satisfies neither sequential nor exact rigidity.
\end{example}

\begin{example}[Incompatible sets] \label{ex:incompatible}
    A notion that is closely connected to rigidity is incompatibility.
    Following \cite{MR3401010}, we call two compact sets $K_1, K_2 \subset \RRmn$
    \emph{gradient incompatible} if they are disjoint and for every $u \in W^{1,\infty}(Q;\RRm)$ satisfying
    \begin{equation}
        Du \in K_1 \cup K_2,
    \end{equation}
    there exists $i \in \set{1,2}$ such that
    \begin{equation}
        Du \in K_i.
    \end{equation}
    Similarly, we call $K_1$ and $K_2$ \emph{incompatible} (resp.\ \emph{homogeneously incompatible}),
    if they are disjoint, and for every gradient Young measure (resp.\ homogeneous gradient Young measure)
    $\nu \in GY(Q;\RRmn)$ satisfying
    \begin{equation}
        \supp \nu_x \subset K_1 \cup K_2 \qquad \text{for a.e. } x \in Q,
    \end{equation}
    there exists $i \in \set{1,2}$ such that
    \begin{equation}
        \supp \nu_x \subset K_i \qquad  \text{for a.e. } x \in Q.
    \end{equation}
    We call finitely many compact sets $K_1,\dots,K_N \subset \RRmn$ \emph{incompatible} (resp.\ \emph{gradient incompatible}, \emph{homogeneously incompatible}) if the sets $K_i$ and $\bigcup_{j \neq i}K_j$ are pairwise incompatible
    (resp.\ gradient incompatible, homogeneously incompatible).
    The different notions of incompatibility admit characterizations analogous to \autoref{prop:seqRig}, see \cite{MR3401010}.

    Furthermore, from Proposition \ref{prop:seqRig} it follows that if $K_1,...,K_N \subset \RRmn$ are disjoint and compact,
    their union satisfies exact rigidity (resp.\ sequential rigidity) if and only if
    $K_1,...,K_N$ are gradient incompatible (resp.\ incompatible) \emph{and} each of the sets
    $K_i$ satisfies exact rigidity (resp.\  sequential rigidity).

    \textsc{Ball and James} \cite{MR3401010} were able to characterize incompatibility in terms of certain transition layer estimates being valid,
    and in particular they showed that if compact sets $K_1,...,K_N$ are incompatible, then there exists $\eps >0$ such that the closed $\eps$-neighborhoods of $K_1,\dots,K_N$ are still incompatible.
    Building on earlier versions of their work, \textsc{De Lellis and Székelyhidi} \cite{MR2253059} proved that incompatibility of two compact disjoint sets $K_1, K_2 \subset \RRmn$  already implies the quantitative estimate
    \begin{equation} \label{eq:DeLellisSzekelyhidi}
        \min \set{\norm{\dist_{K_1}(Du)}_{L^p},\norm{\dist_{K_2}(Du)}_{L^p}} \leq C(K_1,K_2,p,\Omega)\norm{\dist_{K_1 \cup K_2}(Du)}_{L^p}
    \end{equation}
    for every $u \in W^{1,p}(\Omega;\RRm)$, $p \in (1,\infty)$, where $\Omega \subset \RRn$ is any open, bounded and connected set with Lipschitz boundary.
\end{example}

The following example is closely related to rigidity of $\son$.
\begin{example}[Conformal matrices] \label{ex:Conformal}
    Let $\conf$ denote the set of \emph{conformal matrices}, 
    \begin{equation}
        \conf = (0,\infty)\son = \set[\Big]{A \in \RRnn \given A^TA = \frac{\abs{A}^2}{n}\idn,\, \det A >0}.
    \end{equation}
    Note that by the arithmetic-geometric mean inequality applied to the singular values of $A$, it follows
    \begin{equation} \label{eq:confqc}
        \conf \cup \set{0} = \set[\big]{A \in \RRnn \given \abs{A}^n - n^\frac{n}{2}\det A = 0} = f^{-1}(0),
    \end{equation}
    where the function $f(A)\coloneq \abs{A}^n - n^\frac{n}{2}\det A$ is non-negative and strictly quasiconvex.
    In particular, by \autoref{rem:UnifQCsets}, $\conf$ satisfies homogeneous sequential rigidity.
    However, $\conf$ does not satisfy exact rigidity; in particular, $\conf$ is not sequentially rigid.
    Indeed, for $n=2$, the differential inclusion $Du \in \mathrm{CO}_+(2)$ is equivalent to $u$ satisfying the Cauchy-Riemann equations and thus,
    after identifying $\CC$ with $\RR^2$, every holomorphic function
    $u$ satisfies $Du \in \mathrm{CO}_+(2)$.
    For $n\geq 3$, we still have that
    every \emph{special Möbius transformation},
    \begin{align}
        \phi(x) = A x + b \quad \text{ or } \quad  \phi(x) & =  A E \frac{x-a}{\abs{x-a}^2} + b, \qquad A \in \conf, \, a,b \in \RRn, \\
        E                                                  & = \diag(-1,1,...,1),
    \end{align}
    satisfies $D\phi \in \conf$. The special Möbius transformations form an $(n+1)(n+2)/2$-dimensional Lie group under composition, which we call the \emph{special Möbius group} and denote by $\cM_n$.
    Here, \emph{special} refers to the fact that we only consider Möbius transformations $\phi$ that are orientation preserving (i.e., satisfy $\det D \phi >0$).

    If $n\geq 3$, then $\conf$ is still rigid in the sense that by Liouville's theorem on conformal mappings, the special Möbius transformations are the \emph{only smooth} functions satisfying $Du \in \conf$.
    \textsc{Reshetnyak} \cite{MR218544} proved that this still holds if we only assume $u \in W^{1,n}(\Omega;\RRn)$, extending earlier results of \textsc{Gehring} \cite{MR125964}.
    This was further improved by \textsc{Iwaniec} \cite{MR1189867}, who showed that there exists a critical threshold $p=p(n) < n$ for which any
    map $u \in W^{1,p}(\Omega;\RRn)$ satisfying $Du \in \conf$ must be a Möbius transformation, while the same fails for every $q < p(n)$.
    {\sc Iwaniec and Martin} \cite{MR1859913} were able to show that $p(n)= \frac{n}{2}$ whenever $n$ is even, and it is conjectured that the same holds in odd dimensions.

    Another consequence of $\conf$ being homogeneously sequentially rigid is the following: if $K \subset \conf$ is compact and
    $(u_j)_{j \in \NN}$ is a bounded sequence in $W^{1,1}(\Omega;\RRn)$
    satisfying $\dist_K(Du_j) \to 0$ in measure, then up to taking a subsequence, $Du_j \to D\phi$ in measure, where $\phi$ is a special Möbius transformation.

    This is also closely related to \autoref{ex:TangentspaceMoebius}: namely, $V_2$ is the tangent space to $\conf$ at the identity matrix,
    while the polynomials in $\ker \cA^{V_2}$ are exactly the infinitesimal generators of the special Möbius group $\cM_n$.
    In fact, the $\CC$-elliptic estimate \eqref{eq:CC-Ell-Korn} corresponding to $V_2$ was first proved
    by \textsc{Reshetnyak} \cite{MR264464} and played a central role in his work on conformal mappings \cite{MR1326375}.

    Building on this, {\sc Faraco and Zhong} \cite{MR2207734} showed the quantitative rigidity estimate
    \begin{equation} \label{eq:faraco zhong}
        \inf_{\substack{\phi \in \cM_n\\D\phi(\Omega') \subset E}} \norm{Du - D\phi}_{L^2(\Omega')} \leq C\norm{\dist_E(Du)}_{L^2(\Omega)} \qquad \forall u \in W^{1,2}(\Omega;\RRn),
    \end{equation}
    where
    \begin{equation}
        E = \bigcup_{i=1}^N [m_i,M_i]\son, \qquad 0<m_1\leq M_1 <\dots<m_N \leq M_N < \infty,
    \end{equation}
    and $\Omega' \subset \subset \Omega \subset \RRn$ are open and connected.
\end{example}

\begin{example} \label{ex:SeqButNotQuant}
    Consider the set
    \begin{equation}
        K\coloneq \set[\bigg]{
            \begin{pmatrix}
                t & 0   \\
                0 & t^2
            \end{pmatrix} \given t \in [0,1]
        } \subset \RR^{2 \times 2}.
    \end{equation}
    Then $K$ satisfies sequential rigidity, but the tangent space to $K$ at $0$ is clearly spanned by the rank-one direction $e_1 \otimes e_1$.
    In particular, by \autoref{thm:GeneralRigidity}, $K$ does not satisfy $L^p$ quantitative rigidity for any $p \in (1,\infty)$.

    Indeed, if $u\in W^{1,\infty}(Q;\RR^2)$ satisfies $Du \in K$, it must be of the form $u_1 = u_1(x_1)$ and $u_2 = u_2(x_2)$.
    Then
    \begin{equation}
        (\del_1 u_1(x_1))^2 = \del_2 u_2(x_2) \qquad \text{and} \qquad \del_1 u_1 \geq 0
    \end{equation}
    imply that $Du$ is constant.
    Thus, $K$ satisfies exact rigidity.

    Furthermore, let $\nu \in GY_\Hom(\RR^{2 \times 2})$ be supported in $K$; and denote $[\nu] = B$.
    Then, since the determinant is quasiaffine and $\supp \nu \subset K$,
    \begin{equation}
        \begin{split}
        \int_{\RR^{2 \times 2}} A_{11}^3 \,\dd \nu &=  \int_{\RR^{2 \times 2}} \det A \,\dd \nu = \det B = B_{11}B_{22}                                                                                                      \\
        &=                                             \int_{\RR^{2 \times 2}} A_{11}\, \dd \nu\int_{\RR^{2 \times 2}} A_{22}\, \dd \nu = \int_{\RR^{2 \times 2}} A_{11}\, \dd \nu\int_{\RR^{2 \times 2}} A_{11}^2\, \dd \nu \\
        &\leq                                           \left(\int_{\RR^{2 \times 2}} A_{11}^2\, \dd \nu\right)^\frac{3}{2}
        \end{split}
    \end{equation}
    Since $\abs{\cdot}^\frac{3}{2}$ is strictly convex, this shows that the pushforward
    of $\nu$ under the projection onto the $A_{11}$-component is a Dirac measure.
    By the structure of $K$, it follows that $ \nu$ is a Dirac measure.
    As $\nu$ was arbitrary, it follows  by \autoref{prop:seqRig} that $K$ satisfies sequential rigidity.
\end{example}

Let us now move on to examples of sets that satisfy $L^p$ quantitative rigidity.
Note that for each of the sets considered below, it follows from \autoref{thm:Stability} that any sufficiently small graphical perturbation of the set again satisfies $L^p$ quantitative rigidity.

Furthermore, if $H \subset \RRmn$ satisfies $L^p$ quantitative rigidity and $A_1 \in  \GL(m)$, $A_2 \in \GL(n)$ and $A_3 \in \RRmn$, 
then by a change of coordinates and \autoref{prop:domain invariance}, it follows that $A_1 H A_2 + A_3$ satisfies $L^p$ quantitative rigidity as well.
However,
\begin{equation}
    \dim(\GL(m)\times\GL(n)) = m^2 + n^2,
\end{equation}
which for $m,n \geq 2$ is relatively small compared to the dimension of the space of linear maps on $\RRmn$,
\begin{equation}
    \dim \Lin(\RRmn;\RRmn) = m^2n^2.
\end{equation}
\begin{example}
    We begin by explaining how $L^p$ quantitative rigidity of $\son$ follows from \autoref{thm:GeneralRigidity}.
    Condition \ref{intro(a)}, that is, sequential rigidity of $\son$, was first shown by \textsc{Reshetnyak} \cite{MR218544}.
    Since $T_\id\son = \RRnnsk$, condition \ref{intro(b)} is by \autoref{cor:Subspace-QuantIffExact} equivalent to Korn's inequality.
\end{example}
A related example is the following.
\begin{example}
    In nonlinear elasticity theory, especially in the modelling of solid-solid phase transitions (see, e.g., \cite{MR1731640}), one is often interested in the rigidity (or lack of rigidity) of sets of the type $K=\bigcup_{i=1}^N \son A_i$.
    Let $A_i \in \GL(n)$, $i=1,\dots,N$ such that the sets
    \begin{equation}
        K_i \coloneq  \son A_i
    \end{equation}
    are incompatible in the sense of \autoref{ex:incompatible}.
    Then,
    \begin{equation}
        K \coloneq  \bigcup\nolimits_{i \leq N}  K_i
    \end{equation}
    again satisfies the conditions of \autoref{thm:GeneralRigidity}, see \autoref{ex:incompatible}.
    Hence, $K$ satisfies $L^p$ quantitative rigidity.
    This is originally due to \textsc{De Lellis and Székelyhidi} \cite{MR2253059}, who proved that the estimate for the set $K$
    can be reduced to the rigidity estimate for $\son$ \cite{FJM}.
    In general, it can be quite hard to show incompatibility of the sets $K_i$,
    but there are some known examples, including the following:
    \begin{enumerate}
        \item The sets $\son$ and $\son H$, where $H = \diag(\mu_1,\dots,\mu_n)\subset \RRnn$ is a diagonal matrix satisfying $\mu_i >0$ for $i = 1,\dots,n$ and
            \begin{equation} \label{eq:defH}
                \sum_{i = 1}^{n}(1 - \mu_i)(1 - \det H / \mu_i) > 0.
            \end{equation}
            This was first shown by \textsc{Matos} \cite{MR1156593}.
            Building on his work, \textsc{Müller and Chaudhuri} \cite{MR2039459} proved quantitative rigidity in this case. 
        \item The sets $K_i = \lambda_i \son$, where $0<\lambda_1< \cdots < \lambda_N$.
            Here, incompatibility follows as a special case of the result by \textsc{Faraco and Zhong} \cite{MR2207734} on geometric rigidity for conformal matrices.
    \end{enumerate}
\end{example}

The following example is new and extends the result of \textsc{Lamy, Lorent and Peng} \cite{MR4735627}.

\begin{proposition} \label{prop:null lagrangian ex}
    Let $K \subset \RRmn$ be a compact $1$-dimensional $C^1$-submanifold.
    If there exists a $2$-homogeneous null Lagrangian  (\ie quasiaffine function) $f \in C^{\infty}(\RRmn)$ such that
    \begin{equation} \label{eq:NullLagrangianCond}
        f(A - B) \geq \abs{A-B}^2 \qquad \forall A,B \in K,
    \end{equation}
    then $K$ satisfies $L^p$ quantitative rigidity.
\end{proposition}
\begin{proof}
It suffices to show that $K$ satisfies conditions \ref{intro(a)} and \ref{intro(b)} of \autoref{thm:GeneralRigidity}.

\Step{1}
Let us first check condition \ref{intro(b)} on the tangent spaces to $K$.
By \autoref{lem:1d-Subspace}, we only need to show that for each $A \in K$, the space $T_AK$ is not spanned by a rank-one matrix.
Now let $A_0 \in K$, and let $A(\cdot) \in C^1([0,1);K)$ so that $A(0)=A_0$ and $\dot A(0) \in T_{A_0}K \setminus \set{0}$, 
where we use the notation $\dot A(t) \coloneq  \ddt A(t)$.
We now argue by contradiction,
so assume that $\rank \dot A(0) = 1$. 
Since $f$ is quasiaffine and therefore also rank-one affine, it follows that
\begin{equation} 
    D^2f(0) \bigl[\dot A(0), \dot A(0)\bigr] = 0,
\end{equation}
and since $f$ is $2$-homogeneous, we further have
\begin{equation}
    Df(0) = 0, \qquad f(0) = 0.
\end{equation}
But then, since $f$ is smooth and $(A(t) - A_0)\otimes (A(t) - A_0) = t^2\dot A(0) \otimes \dot A(0) + o(t^2)$,
\begin{equation}
        f(A(t)-A_0) = \frac{1}{2}D^2f(0)\bigl[A(t)-A_0,A(t)-A_0\bigr] + Df(0)[A(t)-A_0] +f(0) + O(t^3) = o(t^2),
\end{equation}
contradicting \eqref{eq:NullLagrangianCond}.

\Step{2} We show that every homogeneous gradient Young measure supported in $K$ is trivial.
This argument is due to \textsc{\v{S}verák} \cite{MR1246459}.
Assume that $\nu \in GY_\Hom(\RRmn)$ satisfies $\supp \nu \subset K$.
Then by \eqref{eq:NullLagrangianCond} and since $f$ is quasiaffine, it follows
\begin{align}
    \int_\RRmn \int_\RRmn \abs{A-B}^2\,\dd \nu(A) \,\dd \nu(B) &\leq \int_\RRmn \int_\RRmn f(A-B)\,\dd \nu(A) \,\dd \nu(B)\\
    &= \int_\RRmn f([\nu]-B) \,\dd \nu(B) = f([\nu]-[\nu]) = 0,
\end{align}
implying that $\nu$ is a Dirac measure.

\Step{3} We show that $K$ satisfies exact rigidity.
Let $u \in W^{1,\infty}(Q;\RRm)$ satisfy $Du \in K$.
We aim to show that $u$ is affine.
First, as in the proof of \cite[Thm.~3]{MR1246459}, equation \eqref{eq:NullLagrangianCond} implies that $u \in W^{2,2}_\loc(Q;\RRm)$ and
\begin{equation} \label{eq:W22 loc est}
    \int_{B(x,r/2)} \abs{D^2u}^2\dx\leq \frac{C}{r^2} \int_{B(x,r)} \abs{Du - A}^2 \dx
\end{equation}
for all $A \in \RRmn$ and $B(x,r) \subset Q$.
Furthermore, if $Du$ is continuous at some point $x \in Q$, then by Step 1 and \autoref{cor:smallOsc} it follows that $Du$ is constant in a neighborhood of $x$.
In particular, it holds $D^2u(x) = 0$ at a.e.\ such $x$.
Thus, it suffices to show that $Du$ is continuous at almost every $x\in Q$.
This is a consequence of the excess decay estimate proven in \autoref{lem:excessdecay} below, see, for example, \cite[pp.~95-96]{MR717034}.
\end{proof}

\begin{lemma} \label{lem:excessdecay}
    Let $K \subset \RRmn$ be a compact $1$-dimensional $C^1$-submanifold satisfying \eqref{eq:NullLagrangianCond}.
    Let $u \in W^{1,\infty}(Q;\RRm)$ satisfy $Du \in K$ and $B(x,r) \subset Q$,
    and denote by
    \begin{equation}
        U(u;x,r) \coloneq \dashint_{B(x,r)}\abs{Du - (Du)_{B(x,r)}}^2\dx
    \end{equation}
    the \emph{excess of $Du$ over $B(x,r)$}.
    Then for $\tau \in (0,1/2)$, there exists a constant $\eps_0=\eps_0(K,\tau)>0$ such that the following holds.
    If
    \begin{equation}
        U(u;x,r) \leq \eps_0,
    \end{equation}
    then
    \begin{equation} \label{eq:excessdecay}
        U(u;x,\tau r) \leq  \tau^2 U(u;x,r).
    \end{equation}
\end{lemma}
\begin{proof}
Assume that the assertion of the lemma fails.
Then there exist sequences $(u_k)_{k \in \NN} \subset W^{1,\infty}(Q;\RRm)$ and $(B_k=B(x_k,r_k))_{k\in\NN} $ satisfying $Du_k \in K$,
\begin{equation} \label{eq:def lambda_k}
    U(u_k;x_k,r_k) \eqcolon \lambda_k^2 \to 0
\end{equation}
and
\begin{equation}
    U(u_k;x_k,\tau r_k) >  \tau^2 U(u_k;x_k,r_k).
\end{equation}
Choose $A_k \in K$ so that
\begin{equation} \label{eq: (Du_k) - A_k}
    \abs{(Du_k)_{B_k}-A_k} = \dist_K((Du_k)_{B_k}) \leq \lambda_k,
\end{equation}
and denote
\begin{equation}
    V_k \coloneq T_{A_k}K.
\end{equation}
After possibly taking a subsequence, we may assume that $A_k \to A \in K$ as $k \to \infty$ and therefore also $V_k \to V \coloneq T_AK$.
Now, define
\begin{equation}
    v_k(x) \coloneq \frac{u_k(r_kx + x_k)-r_kA_kx}{\lambda_k r_k},
\end{equation}
which satisfies $v_k \in W^{1,\infty}(B(0,1);\RRm)$,
\begin{equation} \label{eq:excess is one}
    U(v_k;0,1) = 1
\end{equation}
and
\begin{equation} \label{eq:excess B tau}
    U(v_k;0,\tau) > \tau^2.
\end{equation}
We now show that $\dist_V(Dv_k) \to 0$ in measure.
Let $\eps \in (0,1)$ be arbitrary. 
Since $K$ is a $C^1$-submanifold, by \autoref{lem:PropertiesMFLD}~\ref{lem:PropertiesMFLD(ii)} it follows that
\begin{equation}
    \delta \coloneq \inf_{A \in K} \overline \delta_K(A,T_AK,\eps^2/4)
\end{equation}
is strictly positive.
Thus, we can choose $N \in \NN$ sufficiently large so that 
\begin{equation} \label{eq:dist V V_k}
    \frac{\lambda_k}{\eps}   \leq \frac{\delta}{2}  \quad \text{and} \quad d(V_k,V) \leq \frac{\eps^2}{4} \qquad \text{for all } k \geq N.
\end{equation}
Now let
\begin{align}
    S_k &\coloneq \set{x \in B_k \given \abs{Du_k(x) - (Du_k)_{B_k}}> \lambda_k/\eps},\\
    T_k &\coloneq r_k^{-1}(S_k-x_k),
\end{align}
so that by \eqref{eq:def lambda_k}
\begin{equation}
    \abs{S_k} \leq \frac{\abs{B_k}U(u_k;x_k,r_k)}{(\lambda_k/\eps)^2} = \eps^2 \abs{B_k}, \qquad \abs{T_k} \leq \eps^2 \abs{B(0,1)} .
\end{equation}
Now, let $k \geq N$, $x \in B(0,1) \setminus T_k$ and set $y \coloneq r_kx + x_k \in B_k \setminus S_k$.
Note that by \eqref{eq: (Du_k) - A_k},
\begin{equation} \label{eq:Du_k - A_k}
    \abs{Du_k(y)-A_k} \leq \abs{Du_k(y)-(Du_k)_{B_k}} + \abs{(Du_k)_{B_k}-A_k} \leq \lambda_k/\eps + \lambda_k \leq \delta.
\end{equation}
Then, by \eqref{eq:dist V V_k}, \eqref{eq:Du_k - A_k}, and by definition of $\delta$, it holds
\begin{equation}
    \begin{split}
    \dist_V(Dv_k(x)) &\leq \dist_{V_k}(Dv_k(x)) + d(V_k,V)\abs{Dv_k(x)} \\
    &\leq \frac{1}{\lambda_k}\dist_{V_k}(Du_k(y)-A_k) + \frac{\eps^2}{4\lambda_k}\abs{Du_k(y)-A_k}\\
    &\leq \frac{\eps^2}{2\lambda_k}\abs{Du_k(y)-A_k}\\
    &\leq \eps.
    \end{split}
\end{equation}
Hence, for $k$ large enough we have
\begin{equation}
    \abs{\set{x \in B(0,1) \given \dist_V(Dv_k(x)) > \eps}} \leq \abs{T_k} \leq \eps^2\abs{B(0,1)}.
\end{equation}
As $\eps > 0$ was arbitrary, this concludes the proof that $\dist_V(Dv_k) \to 0$ in measure.
By Step 1 of the proof of \autoref{prop:null lagrangian ex}, we know that $V$ satisfies sequential rigidity,
from which it follows that $Dv_k$ converges in measure to some matrix $F \in V$ (see, e.g., \cite[Lem.~4.12]{Rindl} and \autoref{prop:seqRig}).
From \eqref{eq:W22 loc est} and \eqref{eq:excess is one}, it further follows that
\begin{equation}
    \int_{B(0,1/2)} \abs{D^2v_k}^2 \dx \leq C\int_{B(0,1)} \abs{Dv_k - (Dv_k)_{B(0,1)}}^2 \leq C U(v_k;0,1) = C,
\end{equation}
hence $Dv_k \to F$ strongly in $L^2(B(0,1/2))$.
This, however, contradicts \eqref{eq:excess B tau}.
\end{proof}
\section*{Acknowledgements} 
This paper is based on the author's master's thesis, which he completed at the University of Bonn. 
The author wishes to thank Konstantinos Zemas and Sergio Conti for supervising his thesis and for introducing him to this topic. 
He also thanks Stefan Müller for several helpful discussions and Manfred Lehn for giving him an introduction to the algebraic geometry involved in understanding $\CC$-ellipticity.
\appendix
\section{Extension to general domains} \label{sec:generalDomains}

Let $p \in [1,\infty)$.
Following \textsc{Hurri-Syrj\"anen} \cite{ImprovedPoincare}, we denote by $\sP(p,p,1)$ the class of all open and bounded sets $\Omega \subset \RRn$
satisfying the \emph{weighted Poincaré inequality}
\begin{equation} \label{eq:weighted poincaré}
    \norm{v - (v)_\Omega}_{L^p(\Omega)} \leq C(\Omega) \norm{\dist_{\del \Omega}(x) Dv}_{L^p(\Omega)} \qquad \forall v \in W^{1,p}(\Omega).
\end{equation}
By \cite{ImprovedPoincare}, the class $\sP(p,p,1)$ contains all John domains, and in particular all bounded Lipschitz domains.

The aim of this appendix is to show the following.
\begin{proposition}\label{prop:domain invariance}
    Let $H \subset \RRmn$ be closed.
    Assume that there exist open sets $U \subset U' \subset \subset \RRn$ and a constant $C_*> 0$ such that
    $H$ satisfies the interior estimate
    \begin{equation} \label{eq:int est}
        \norm{Du - (Du)_{U}}_{L^p(U)} \leq C_* \norm{\dist_H(Du)}_{L^p(U')} \qquad \forall u \in W^{1,p}(U';\RRm).
    \end{equation}
    Then, for every $\Omega \in \sP(p,p,1)$ it holds
    \begin{equation} \label{eq:glob est}
        \norm{Du - (Du)_{\Omega}}_{L^p(\Omega)} \leq C_*C(p,\Omega,U,U') \norm{\dist_H(Du)}_{L^p(\Omega)} \qquad \forall u \in W^{1,p}(\Omega;\RRm).
    \end{equation}
\end{proposition}
The proof of \autoref{prop:domain invariance} is mostly contained in the following lemma.
While this construction is certainly not new and appears similarly for example in \cite{FJM}, the author was unable to trace down a reference for the formulation presented below.
\begin{lemma} \label{lem:local to global}
    Let $\Omega \in \sP(p,p,1)$, where $p \in [1,\infty)$.
    Assume that we are given two functions $v \in L^p(\Omega;\RRM)$, $V \in L^p(\Omega;[0,\infty))$, and a parameter $\theta \in (0,1]$, such that for every pair of concentric cubes
    $Q(x,\theta r) \subset Q(x,r) \subset \subset \Omega$ it holds
    \begin{equation} \label{eq:loc to glob ass}
        \norm{v - (v)_{Q(x,\theta r)}}_{L^p(Q(x,\theta r))} \leq \norm{V}_{L^p(Q(x,r))}.
    \end{equation}
    Then
    \begin{equation} \label{eq:loc to glob}
        \norm{v - (v)_{\Omega}}_{L^p(\Omega)} \leq C(p,\Omega,\theta) \norm{V}_{L^p(\Omega)}.
    \end{equation}
\end{lemma}
\begin{proof}[Proof that \autoref{lem:local to global} implies \autoref{prop:domain invariance}]
Assume that the hypotheses of \autoref{prop:domain invariance} are satisfied.
Then there exist $x_0 \in \RRn$ and $r_0,\theta > 0$ such that
\begin{equation}
    Q(x_0,\theta r_0) \subset U \subset U' \subset Q(x_0,r_0).
\end{equation}
Let $u \in W^{1,p}(\Omega;\RRm)$ and set $v := Du$ and $V:=2C_*\dist_H(Du)$.
Then \eqref{eq:loc to glob ass} follows from the fact that \eqref{eq:int est} is invariant under scaling and translation of the domains $U,U'$.
By \autoref{lem:local to global}, it follows that \eqref{eq:loc to glob} holds, which is by definition of $v$ and $V$ equivalent to \eqref{eq:glob est}.
\end{proof}
\begin{proof}[Proof of \autoref{lem:local to global}]
Set
\begin{equation}
    \lambda \coloneq  \frac{3}{\theta},
\end{equation}
and let $\set{Q_i}_{i \in \NN} = \set{Q(x_i,r_i)}_{i \in \NN}$ be a Whitney cover of $\Omega$, satisfying
with $Q'_i \coloneq  Q(x_i,\lambda r_i)$
\begin{enumerate}
    \item \label{local global(i)} $\overline \Omega = \bigcup_{i \in \NN}\overline{Q_i}$, and the union is disjoint up to null sets;
    \item \label{local global(ii)} $\sum_{i \in \NN} \chi_{Q'_i} \leq N \chi_\Omega$;
    \item \label{local global(iii)} The radii of the cubes are comparable to $\dist_{\del \Omega}(x)$ in the sense that
        \begin{equation} 
            \lambda r_i \leq \dist_{\del \Omega}(x) \leq C\lambda r_i \qquad \forall x \in Q'_i, \; i \in \NN.
        \end{equation}
\end{enumerate}
The constants $N$ and $C$ can be chosen to depend only on $\theta$ and the dimension (cf., e.g., \cite[Ch.~IV]{Stein}).
Furthermore, let $\set{\eta_i}_{i \in \NN}$ be a smooth partition of unity on $\Omega$ satisfying for $i \in \NN$
\begin{enumerate}[resume]
    \item \label{local global(iv)} $\supp \eta_i \subset Q(x_i,2r_i)$;
    \item \label{local global(v)} $\norm{D\eta_i}_{C^0} \leq C r_i^{-1}$.
\end{enumerate}
Now, for $i \in \NN$, set $v_i \coloneq  (v)_{Q(x_i,3r_i)}$ and recall that by \eqref{eq:loc to glob ass} it holds
\begin{equation} \label{eq:BDv-vi}
    \norm{v-v_i}_{L^p(Q(x_i,3r_i))} \leq \norm{V}_{L^p(Q_i')}.
\end{equation}
Let $\tilde v \in C^\infty(\Omega;\RR^M)$ be given by
\begin{equation}
    \tilde v(x)  \coloneq  \sum_{i \in \NN} v_i\eta_i(x),
\end{equation}
and observe that by \ref{local global(ii)}, \ref{local global(iv)} and \eqref{eq:BDv-vi},
\begin{equation} \label{eq:v - tilde v}
    \begin{split}
    \norm{v-\tilde v}^p_{L^p(\Omega)} & \leq N^{p-1}\sum_{i \in \NN} \norm{\eta_i (v-v_i)}_{L^p}^p  \\
    & \leq N^{p-1}\sum_{i \in \NN} \norm{V}_{L^p(Q_i')}^p \leq N^p \norm{V}_{L^p(\Omega)}^p.
    \end{split}
\end{equation}
As a next step, we derive a bound on $D\tilde v$.
To this end, note that if for some $i,j \in \NN$ we have
\begin{equation}
    Q_i \cap Q(x_j,2r_j) \neq \emptyset,
\end{equation}
then by \ref{local global(iii)} and \ref{local global(iv)} it follows
\begin{equation}
    \abs[\big]{Q(x_i,3r_i) \cap Q(x_j,3r_j)} \geq \frac{\max\set{r_i^n,r_j^n}}{C},
\end{equation}
and thus by the triangle inequality and \eqref{eq:BDv-vi},
\begin{equation} \label{eq:vi-vj}
    \begin{split}
    \abs{v_i -v_j}^p & \leq 2^{p-1}\dashint_{Q(x_i,3r_i) \cap Q(x_j,3r_j)}  \abs{v(x) - v_i}^p + \abs{v(x)- v_j}^p \dx \\
    &\leq   C(p)   \Bigl(r_i^{-n}{}\norm{V}^p_{L^p(Q'_i)}+r_j^{-{n}}\norm{V}^p_{L^p(Q_j')}\Bigr).
    \end{split}
\end{equation}
Now let $x \in Q_i$.
Then by \ref{local global(iv)} and since $\sum_j D\eta_j = 0$, we have
\begin{equation}
    D\tilde v(x) = \sum_{j \in \NN:\, x \in Q(x_j,2r_j)} (v_j - v_i)D\eta_j(x),
\end{equation}
so that by \ref{local global(iii)}, \ref{local global(v)} and \eqref{eq:vi-vj} it follows
\begin{equation} \label{eq:bd on tilde v}
    \begin{split}
    \abs{D\tilde v(x)}^p & \leq N^{p-1}  \sum_{j \in \NN:\, x \in Q(x_j,2r_j)} \norm{D\eta_j}_{C^0}^p \abs{v_i-v_j}^p \\
    &\leq C(N,p)  \sum_{j \in \NN:\, x \in Q(x_j,2r_j)} r_j^{-p} \Bigl(r_i^{-n}{}\norm{V}^p_{L^p(Q'_i)}+r_j^{-{n}}\norm{V}^p_{L^p(Q_j')}\Bigr) \\
    & \leq C(N,p)  \sum_{j \in \NN:\, x \in Q(x_j,2r_j)} r_j^{-p-n} \norm{V}^p_{L^p(Q_j')}
    \end{split}
\end{equation}
By \eqref{eq:bd on tilde v} and \ref{local global(iii)}, we thus have
\begin{equation}
    \begin{split}
    \int_{Q_i}\dist_{\del \Omega}(x)^p\abs{D \tilde v (x)}^p\dx \leq C(N,p) \int_{Q_i} \sum_{j:\, x \in Q(x_j,2r_j)} r_j^{-n} \norm{V}^p_{L^p(Q_j')} \dx.
    \end{split}
\end{equation}
Summing over $i$ and exchanging the order of summation, we obtain
\begin{align}
    \int_{\Omega}\dist_{\del \Omega}(x)^p\abs{D \tilde v (x)}^p\dx  &\leq C(N,p) \sum_{j \in \NN} \norm{V}_{L^p(Q_j')}^pr_j^{-n}\int_{ Q(x_j,2r_j)} \sum_{i:\,x \in Q_i} 1 \dx \\
    &\leq C(N,p)\norm{V}_{L^p(\Omega)}^p.
\end{align}
The result now follows from the weighted Poincaré inequality \eqref{eq:weighted poincaré}, which holds by assumption, and from \eqref{eq:v - tilde v}.
\end{proof}

\end{document}